\documentclass{amsart}
\usepackage{amstext, amssymb, amsthm, amsfonts, latexsym,bbm}

\usepackage{graphicx}
\usepackage[utf8]{inputenc}
\usepackage[english]{babel}
\usepackage{enumitem}
\usepackage[version=3]{mhchem}
\usepackage{mathtools}
\usepackage{varioref}

\usepackage{todonotes}
\usepackage{amsmath}

\usetikzlibrary{arrows,automata}
\usepackage{tikz}
\usetikzlibrary{shapes,arrows}
\usetikzlibrary{intersections}

\usepackage{hyperref}

\usepackage[normalem]{ulem}

\newtheorem{theorem}{Theorem}[section]
\newtheorem{lemma}[theorem]{Lemma}
\newtheorem{corollary}[theorem]{Corollary}
\newtheorem{proposition}[theorem]{Proposition}

\theoremstyle{definition}
\newtheorem{definition}[theorem]{Definition}
\newtheorem{example}[theorem]{Example}

\theoremstyle{remark}
\newtheorem{remark}[theorem]{Remark}

\theoremstyle{Conjecture/open problem}

\theoremstyle{assumption}
\newtheorem{assumption}{Assumption}

\newcommand{\1}{\mathbbm{1}}

\newcommand{\N}{\mathbb{N}}

\newcommand{\R}{\mathbb{R}}
\newcommand{\Z}{\mathbb{Z}}

\newcommand{\mF}{\mathcal{F}}
\newcommand{\mG}{\mathcal{G}}
\newcommand{\mN}{\mathcal{N}}
\newcommand{\mV}{\mathcal{V}}
\newcommand{\mE}{\mathcal{E}}
\newcommand{\mU}{\mathcal{U}}
\newcommand{\Eu}{M}

\newcommand{\st}{\,\mid \,}
\newcommand{\lam}{\lambda}
\newcommand{\lamm}{\tau}

\newcommand{\orl}{\overline{r}}
\newcommand{\eps}{\epsilon}
\newcommand{\I}{\mathcal{I}}

\DeclareMathOperator*{\too}{\to}

\renewcommand{\k}{\kappa}

\def\p{^{\prime}}

\newcommand{\mC}{\mathcal{C}}

\newcommand{\mO}{\mathcal{O}}

\newcommand{\mR}{\mathcal{R}}
\newcommand{\mS}{\mathcal{S}}

\newcommand{\supp}{{\rm supp\,}}

\newcommand{\srcsize}{\@setfontsize{\srcsize}{5pt}{5pt}}

\newcommand{\mRa}{\mathcal{R}_{\mathcal{U}}} 
\newcommand{\mRb}{\mathcal{R}_{\mathcal{U}}'} 
\newcommand{\mRu}{(\mathcal{R}_{\mathcal{U}}\cup \mathcal{R}_{\mathcal{U}}')} 
\newcommand{\mRf}{\mathcal{F}} 

\newcommand{\mGuf}{\mathcal{G}_{\mathcal{U},\mathcal{F}}} 
\newcommand{\mEuf}{\mathcal{E}_{\mathcal{U},\mathcal{F}}}
\newcommand{\mVuf}{\mathcal{V}_{\mathcal{U},\mathcal{F}}}
\newcommand{\mWu}{W_{\mathcal{U},\mathcal{F}}} 

\newcommand{\mGr}{(\mathcal{C}_{\mathcal{U},\mathcal{F}},\mathcal{R}_{\mathcal{U},\mathcal{F}})}
\newcommand{\mNuf}{\mathcal{N}_{\mathcal{U},\mathcal{F}}}
\newcommand{\wRn}{\mathcal{R}_{\mathcal{U},\mathcal{F}}} 

\newcommand{\mRfu}{\mathcal{R}_{\mathcal{U},\mathcal{F}}} 

\newcommand{\w}{\Gamma}

\begin{document}

\title[Fast reactions with non-interacting species in SRNs]{Fast reactions with non-interacting species in stochastic reaction networks}
\subjclass[2010]{60J27,60J28,92C42,92B05,92E20}
\keywords{Stochastic reaction networks, mass-action system, Markov process, Continuous-time Markov process, reduction, singular perturbation}


\date{}

\author{Linard Hoessly \and Carsten Wiuf}

\address{Department of Mathematical Sciences, University of Copenhagen, Denmark}
\email{hoessly@math.ku.dk \and wiuf@math.ku.dk}

\begin{abstract}
We consider stochastic reaction networks modeled by continuous-time Markov chains. Such reaction networks often contain many reactions, potentially occurring at different time scales, and have unknown  parameters (kinetic rates, total amounts). This makes their analysis complex. We examine stochastic reaction networks with non-interacting species that often appear in examples of interest (e.g. in the two-substrate Michaelis Menten mechanism). Non-interacting species typically appear as intermediate (or transient) chemical complexes that are  depleted at a fast rate. 
We embed the Markov process of the reaction network into a one-parameter family under a two time-scale approach, such that molecules of non-interacting species are degraded fast. We derive simplified reaction networks where the non-interacting species are eliminated and that approximate the scaled Markov process in the limit as the parameter becomes small.
Then, we derive sufficient conditions for such reductions based on the reaction network structure for both homogeneous and time-varying stochastic settings, and study examples and properties of the reduction.
 
  \end{abstract}
  
  \maketitle

\section{Introduction}

Reaction network  theory  offers a quantitative framework for biochemistry, systems biology, and cellular biology, by enabling the modeling of biological systems. Deterministic models  have been the main focus area with contributions going back more than hundred years \cite{Murray2002}. However, the increasing interest in living systems at the cellular level motivates the use of stochastic models to describe the variation and noise found in   systems with low molecule numbers. Typically, stochastic models are continuous-time Markov chains (CTMCs) on the state space $\Z^n_{\ge0}$, where a state represents the vector of molecule numbers  of the $n$ species in the system.

We are here concerned with the transient behavior of such CTMCs, in contrast to the stationary behavior (the existence of stationary distributions). In practice, variations in molecule numbers and reaction rates might yield phenomena that evolve on different time-scales, enabling simplifications \cite{Weinan}. In particular, we are interested in systems with two time-scales, also known as slow-fast systems, where a set of reactions are fast (in a relative sense) compared to the remaining (slow) reactions. The objective is to approximate, in a mathematically rigorous way, the dynamics of the original CTMC, with another CTMC in smaller dimension (with fewer species). This other CTMC should ideally be interpreted as a  reaction network, obtained by reduction of the original reaction network. We will give conditions for when this can be done.

In the deterministic setting,   the \emph{heuristic} quasi-steady state approximation (QSSA) \cite{segal2} and singular perturbation theory in the sense of Tikhonov-Fenichel \cite{tikh,fenichel} have been the main means to derive lower dimensional reduced models. (See \cite{hta,gwz3} and references therein for the relationship between the two approaches and when the QSSA is valid.)
A motivation for our work is the situation for which reactions involving so-called \emph{non-interacting species} \cite{Variable_el,saez2} in the reactant has reaction rate constants scaled by $1/\epsilon$, where $\epsilon>0$ is a small number \cite{walcher19}. We will consider a \emph{similar} situation for stochastic reaction networks and point out differences between the two settings.

 Many stochastic studies consider  physical or heuristic-based derivations to extract reduced reaction networks, for example, SDEs or hybrid models, or \emph{ad-hoc} reductions  \cite{Schnoerr_2017,othmer2,Gorban_ov}. One stream of research eliminates species using heuristic projection arguments \cite{Janssen2,Janssen3}. Rigorous simplifications and reductions  often follow scaling limits of Markov processes  in a multi-scale setting \cite{approx_kurtz}. These might be applied to concrete examples, if certain conditions are satisfied and the different scaling parameters are  balanced (in a specific sense) \cite{BKPR06,KK13,PP15}. Scaling laws for a special class of reaction networks with \emph{intermediate species} (a special type of non-interacting species)  and their explicit reduction are given in \cite{CW16}. We extends this work to reaction networks with non-interacting species on two time-scales.

These two time-scales separate the transition intensities of the CTMC into a fast and a slow component such that the $Q$-matrix has the following form
$$Q^\eps=\frac{1}{\eps}\widetilde{Q}+\widehat{Q}.$$
The transition intensities of the reactions in the network are divided into two kinds:
 \begin{itemize}
\item Fast reactions with scaled transition intensity $\lam^\eps_{y \to y\p}(x):=\frac{1}{\eps}\lam_{y \to y\p}(x)$. 
\item Slow reactions with unscaled  transition intensity $\lam^\eps_{y \to y\p}(x):=\lam_{y \to y\p}(x)$.
\end{itemize}
The fast reactions are determined by non-interacting species. 
Our work is inspired by previous work on non-interacting species in deterministic systems \cite{Variable_el,saez2,walcher19} and intermediate species in stochastic models \cite{CW16}. The technical part is based on singularly perturbed Markov chains  \cite{yin_zhang_trans,yin2012continuous} and  watched Markov chains \cite{Freedman_a}. In order to derive the limiting dynamics systematically, we associate a graph  to the reaction network that captures the dynamics of the fast reactions. This enables the definition of the reduced reaction network and its dynamics. We give limit theorems for the approximation of  the original CTMC to the  reduced CTMC on compact time intervals. Furthermore, we study the case of time-heterogeneous CTMCs, and show that the same reductions work. 

As an example, consider a \emph{mass-action} reaction network as follows
$$S_1\ce{->[\kappa_1]}U_1+S_2\ce{->[\frac{1}{\eps}\kappa_2]}S_3,\quad U_1\ce{->[\frac{1}{\eps}\kappa_3]}S_4,$$
with two fast reactions determined by the presence of the non-interacting species $U_1$, that degrades fast.
The reduced reaction network is given as
$$
 S_1\ce{->[]}S_3,\quad S_1\ce{->[]}S_2+S_4,
$$
with transition intensities $\frac{\k_1\k_2z_{S_1}(z_{S_2}+1)}{\k_2(z_{S_2}+1)+\k_3}$ for the first reaction and $\frac{\k_1\k_3z_{S_1}}{\k_2(z_{S_2}+1)+\k_3}$ for the second.
After creation of a $U_1$ molecule   (and a $S_2$ molecule), then it might be degraded either by consumption of the $S_2$ molecule, resulting in the net reaction $S_1\ce{->}S_3$, or without consumption of the $S_2$ molecule, resulting in the net reaction $S_1\ce{->}S_2+S_4$. In both cases, the transition intensities reflect that the presence of $S_1$ is required for the reactions to take place.

 We next outline the content, where in $\S$ \ref{notat}  preliminaries on graph theory and reaction networks are covered. In $\S$ \ref{stoch_el}, we introduce non-interacting species and introduce a graph that is used to define the reduced reaction network by elimination of the non-interacting species. In $\S$ \ref{reduced_dynamics}, we study transient approximations for stochastic reaction networks with non-interacting species via the previously introduced reduction. In $\S$ \ref{sectfi}, we give realistic examples, study sufficient conditions for reductions and compare stationary properties of the reduction with the original reaction network. Finally in $\S$ \ref{disc}, we discuss the results, approach, and elaborate on the relation to the literature. In the Appendices $\S$ \ref{additional_vers}, $\S$ \ref{app}, $\S$ \ref{proof_nonhom}, we give proofs as well as brief introductions to the theory on singularly perturbed CTMCs and watched CTMCs.

\section{Preliminaries}\label{notat}
 
\subsection{Notation}
Let $\R^p$ be the real $p$-dimensional space, and $\R^p_{>0}$ ($\R^p_{\geq0}$)  the subset of elements of $\R^p$ with strictly positive (non-negative) entries in all components. A vector $y\in \R^p$ is written as $(y^1,\cdots, y^p)$, where $y^i$ is the $i$-th component.  For vectors $y_1,\cdots , y_k \in \R^p$, $\text{max}(y_1,\cdots,y_k)$ denotes 
the component-wise maximum, and $y_1\ge y_2$ denotes component-wise inequality. The inner product between $y_1$ and $y_2$ is denoted $\langle y_1,y_2 \rangle$. The cardinality of a set $A$ is denoted $|A|$.

\subsection{Graph theory}
 $\mG=(\mV, \mE)$ a directed graph consists of the set of vertices $\mV$ and edge set $\mE$. A directed subgraph of $\mG=(\mV, \mE)$ is a directed graph $\mG\p=(\mV\p, \mE\p)$ with $\mV\p\subseteq \mV, \mE\p\subseteq \mE$ with $ \mE\p$ on $\mV\p$.
A walk is a directed path  $V_{i_1}\rightarrow V_{i_2}\rightarrow \cdots \rightarrow V_{i_{l-1}}\rightarrow V_{i_l}$ (potentially listed as the corresponding sequence of edges).

A multi-digraph $\mG$ is a directed graph where  multiple edges between the same vertices are allowed. In particular, a multi-digraph comes with two functions
\begin{equation*}\label{defmultigr}
s\colon\mathcal{E} \rightarrow \mV\qquad t\colon\mathcal{E} \rightarrow \mV.
\end{equation*}
where the functions 
$s,t$ are the source and target function, respectively.
 Both self-edges (edges $e$ with $t(e)=s(e)$) and parallel edges 
 are possible.

\subsection{Reaction networks (RNs)}

A RN on a finite set $\mS$ is a digraph $\mN=(\mC,\mR)$, where 
 $\mS$ is a finite set of species $\mS=\{S_1,\cdots,S_n\}$, $\mC$ a potentially infinite set of complexes and $\mR$ a potentially infinite set of reactions $\mR=\{r_1,r_2\cdots\}$. 
 Complexes are non-negative linear combinations of species, $y=\sum_{i=1}^ny^i S_i$,  identified with vectors $y=(y^1,\ldots,y^n)$ in $\Z_{\geq 0}^n$. 
 Reactions are directed edges between complexes,  written as $r_i\colon y_i\to y_i\p$ or, generically, as $r\colon y\to y\p $,  potentially omitting $r_i,r$. 
A reaction is said to \emph{consume} the \emph{reactant} $y$ and \emph{create} the \emph{product} $y\p$.  An RN is said to be \emph{finite} if $\mR$ is finite, and otherwise it is \emph{infinite}.

We diverge in two ways from the standard definition of RNs: \emph{trivial} reactions $r\colon y\to y$ (self-loops) are allowed, and the numbers of complexes and reactions are allowed to be infinite. Both extensions are useful when dealing with reduced RNs. From a dynamical point of view, trivial reactions might always be ignored as the dynamics is the same with and without them. Realistic model of bursty gene expression  with infinitely many reactions have been proposed in the literature \cite{CJ20,BA16}. However, our motivation is not to accommodate such examples, but rather to ensure that an RN obtained by reduction of a finite RN is also an RN, finite or infinite. The construction of the reduced RN also holds even if the original RN is infinite. 

\begin{definition}
(i) Two species \emph{interact} if they both appear in a complex of $\mC$. 

(ii) A subset $\mU\subseteq \mS$ is \emph{non-interacting} if it contains no pair of interacting species, and the stoichiometric coefficients of the species in $\mU$ in all complexes are either $0$ or $1$. The species of $\mU$ are said to be non-interacting.

(iii) If $\mU$ is non-interacting, the  species in  $\mS\setminus\mU$ are said to be  \emph{core} species.

\end{definition}  

\begin{example}\label{two_sub_Ing}
Consider a two-substrate Michaelis Menten mechanism \cite[Section 3.1.2]{syst_ing}:
$$r_1\colon E+A\ce{->}EA,\quad r_2\colon EA\ce{->} E+A,\quad r_3\colon EA+B\ce{->}EAB,$$
$$r_4\colon EAB\ce{->} EA+B,\quad r_5\colon EAB\ce{->[]}E+P+Q,$$
where $E$ is an enzyme catalyzing the conversion of two substrates $A, B $ into two  other substrates $P, Q$ by means of transient (or intermediate) steps; here $EA$ and $EAB$ are known as intermediate complexes formed by binding of the molecules in the reactants.
 
This RN has species set $\mS=\{E,A,B,EA,EAB,P,Q\}$ and complex set $\mC=\{E+A,EA,EA+B,EAB,E+P+Q\}$.
The sets  $\mU_1=\{EA,EAB\}$, $\mU_2=\{EA,P\}$ and $\mU_3=\{EAB\}$ are non-interacting. 
\end{example}

In this paper, it is convenient to work with the \emph{directed stoichiometric subspace} of  $\mN=(\mC,\mR)$, defined as 
$$\mathcal{T}=\Big\{\sum_{r\colon y\to y'\in\mR'}\alpha_r( y\p-y)\Big|\alpha_r>0, r\in\mR', \text{ where }\mR'\subseteq\mR\text{ is finite}\Big\}\subseteq\R^n$$ 
(note, this definition allows the RN to be infinite).
For $v\in \R^n$, the set $(v+\mathcal{T})\cap\R^n_{\geq 0}$ defines a \emph{directed stoichiometric compatibility class} of $\mN$.  The RN is said to be \emph{conservative}, respectively,  \emph{sub-conservative}, if there exists a positive vector  $c\in \R^n_{>0}$, such that $\langle c, y'-y\rangle=0$, respectively, $\langle c, y'-y\rangle\le 0$ for any reaction $r\colon y \to y\p\in \mR$.
A sub-conservative RN has compact directed stoichiometric compatibility classes (but not stoichiometric compatibility classes). For a weakly reversible RN, the directed and the undirected stoichiometric subspaces are the same.

\subsection{Stochastic reaction networks (SNRs)}
An SRN is an RN together with a  CTMC $X(t), t\ge 0$, on $\Z_{\geq 0}^n$, modeling the number of molecules of each species over time. A reaction $r\colon y\to y\p$ \emph{fires}  with transition intensity $\lam_r(x)$,  
in which case the state jumps from $X(t)=x$ to $x+y\p-y$ \cite{book_kurtz_crns}.
The Markov process with transition intensities $\lam_r\colon\Z_{\geq 0}^n\to \R_{\geq 0}$, $r\in\mR$, has  $Q$-matrix 
$$Q(x,x+\xi):=\sum_{r\colon y \to y\p\in \mR\colon y\p-y=\xi}\lam_r(x).$$
For (stochastic) mass-action kinetics, the transition intensity for $r\colon y\to y\p$ is 
\begin{equation*}
\label{int}\lam_r(x)=\k_r\tfrac{(x)!}{(x-y)!}\1_{\{x'\colon x'\geq y\}}(x),\quad x\in\Z_{\ge 0}^n,
\end{equation*}
where $x!:=\prod_{i=1}^nx_i!$, 
and  $\k_r$ is a positive reaction rate constant  \cite{book_kurtz_crns}. If there are infinitely many reactions, we assume   
\begin{equation}\label{eq:lambda}
\sum_{r\in\mR}\lam_r(x)<\infty,\quad\text{for all}\quad x\in\Z^n_{\geq 0},
\end{equation}
 such that the corresponding CTMC is well-defined in the sense that it has no instantaneous jumps \cite{norris}. 

When the reactions are indexed, $\mR=\{r_1,r_2,\ldots\}$, we occasionally write $\lambda_i$ and $\kappa_i$, $i=1,2,\ldots$, for convenience.
The following assumption  holds in particular for stochastic mass-action kinetics.

\begin{assumption}\label{ass1}For all reactions,  $r\colon y\to y\p\in\mR $, the transition  intensity 
 $\lam_r\colon \Z_{\geq 0}^n\to\R_{\geq 0}$ 
satisfies the following
$$\lam_r(x)>0\iff x\geq y.$$
\end{assumption}
Under Assumption \ref{ass1}, $(X(t))_{t\ge 0}$, stays in $\Z^n_{\geq 0}$, if $X(0)=x\in\Z^n_{\geq 0}$. In particular, $(X(t))_{t\ge 0}$, is confined  to the directed stoichiometric compatibility class  $(x+\mathcal{T})\cap\Z^n_{\geq 0}$. 
 Assumption \ref{ass1} is fundamental and enforces the reactions to be compatible with the transition intensities of the CTMC. 

\section{Elimination of non-interacting species through fast reactions}\label{stoch_el}

\subsection{Notation for non-interacting species}

In the following, we focus on RNs and SRNs with non-interacting species.
To fix notation, let $\mU\subseteq \mS$ be a non-interacting subset of species. For simplicity, we let $\mU=\{U_1,\dots ,U_m\}$ and  $\mO=\mS\setminus \mU=\{S_1,\dots,S_p\}$  ($p=n-m$), such that $\mS=\mO\cup \mU=\{S_1,\dots,S_p, U_1,\dots,U_m\}$. We let 
\begin{equation*}
\label{projections}
\qquad \rho_{\mO}\colon\R^n\rightarrow \R^p,
\qquad \rho_\mU\colon\R^n\rightarrow \R^m
\end{equation*}
be the projections onto the first $p$ coordinates and  last $m$ coordinates of $\R^n$, respectively.  Consequently, we denote a state by $x=(z,u)\in\Z^p_{\ge0}\times\Z^m_{\ge0}=\Z^n_{\ge0}$.

Define the following sets of reactions
$$ \mRa=\{r\colon y \to y\p\in \mR \st   \rho_\mU(y)\neq 0\},\quad \mRb=\{r\colon y \to y\p\in \mR \st   \rho_\mU(y\p)\neq 0\},$$
such that $\mRa\cup  \mRb$ are the reactions that involve species in $\mU$.

We consider a subset of \emph{fast} (a terminology to be motivated below) reactions $\mRf\subseteq \mRa$ with the following structural property.

\begin{definition}\label{ass0}
The reactions in $\mRf$ are \emph{proper} w.r.t. $\mU$, that is, any non-interacting species is part of a sequence of reactions in $(\mRb\setminus \mRa)\cup \mRf$ of the  form
$$r_{i_0}\colon y_{i_0}\to y_{i_0}\p,\quad r_{i_1}\colon y_{i_1} \to y_{i_1}\p, \quad \cdots\quad r_{i_l}\colon y_{i_l} \to y_{i_l}\p,$$
with $i_0,\ldots,i_l\in\{1,\ldots,k\}$,  $\rho_\mU(y_{i_0})=\rho_\mU(y_{i_l}\p)=0$ and $\rho_\mU(y_{i_j})=\rho_\mU(y_{i_{j-1}}\p)\neq 0$ for $j=1,\cdots ,l$.
Such a sequence of reactions is called a \emph{fast} chain.
\end{definition}

Hence, any molecule of a non-interacting species can be degraded through a sequence of fast reactions (provided sufficient molecule numbers of core species).

\begin{example} Recall Example \ref{two_sub_Ing}. The set   $\mU_1=\{EA,EAB\}$ is a non-interacting set and  $\mRf=\mRa$ is a set of fast proper reactions. For example, the following is a fast chain:
$$r_1\colon E+A\to EA,\quad r_3\colon EA+B\to EAB,\quad r_5\colon EAB\to E+P+Q.$$
Also, $\mU_2=\{EA,P\}$ is a non-interacting set of species, but   $\mRf=\mRa$ is not a set of fast proper reactions, as there are no fast reactions with $P$ in the reactant.
\end{example}

\subsection{The reduced reaction network}\label{subs_red_stoch}

We introduce a labeled multi-digraph to capture the conversion and creation of non-interacting species through fast chains. This graph is similar to the one introduced  in \cite{saez2}. We later use this graph to define a reduced RN and a reduced SRN by elimination of non-interacting species. 

\begin{definition}\label{stochgraph} 
Let $\mN=(\mC,\mR)$ be an RN on $\mS$, $\mU\subseteq \mS$  a set of non-interacting species, and  $\mRf\subseteq\mR_\mU$ a set of  fast proper reactions. Let $\mGuf=(\mVuf, \mEuf)$ be the  labeled multi-digraph with vertex set
$$\mVuf:=\{*_{in},*_{out}\}\cup\{U_i\st U_i\in\mU\},$$
and edge set 
\begin{align*} 
\mEuf&:=\{*_{in}\ce{->[r]}U_i\st r\colon y \to y\p\in \mRb\setminus\mRa, \rho_\mU(y\p)=U_i\}\bigcup\\
&\quad \{U_i\ce{->[r]}U_j, \st r\colon y \to y\p\in\mF\text{ such that }\rho_\mU(y)=U_i, \rho_\mU(y_\gamma\p)=U_j\}\bigcup\\
&\quad \{U_i\ce{->[r]}*_{out}, \st  r\colon y \to y\p\in \mRf \text{ such that }\rho_\mU(y)=U_i, \rho_\mU(y\p)= 0\}.
\end{align*}
Furthermore, let $L\colon \mEuf\to \{1,\ldots\}$ be the function that maps  $\gamma\in\mEuf$ to the corresponding index of the reaction label, that is, if $\gamma\in\mEuf$ has label  $r_i$, then $L(\gamma)=i$.
\end{definition}

The construction of the graph $\mGuf$ also makes sense in the case of an infinite RN. The vertex set is  finite, but the number of edges between any two vertices might be infinite.

\begin{example}\label{cont_ex}
Consider Example \ref{two_sub_Ing} with
 $\mU=\mU_1=\{EA,EAB\}$ and
 $\mRf=\mRa$.
Then $\mGuf$ is

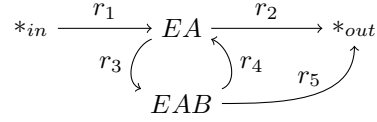
\begin{center}
\begin{tikzpicture}
    \node (p1) at ( 0, 0) {$\ast_{in}$}; 
    \node (p2) at ( 2, 0) {$EA$};
    \node (p3) at ( 4.3,0) {$\ast_{out}$};
    \node (p6) at ( 2,-1) {$EAB$};
    
    \begin{scope}[every path/.style={->}]
       \draw (p1) -- (p2) node[midway,above]{$r_1$};
        \draw (p2) -- (p3) node[midway,above]{$r_2$}; 
        
         \draw[->] (p2) to[out=-160,in=160] node[left]{$r_3$}(p6);   
        \draw[->] (p6) to[out=20,in=-20] node[right]{$r_4$}(p2);   

        \draw (p6) to[out=0,in=-90] node[midway,above]{$r_5$} (p3);        
    \end{scope}  
\end{tikzpicture}.
\end{center}
\normalsize
There are infinitely many walks from $\ast_{in}$ to $\ast_{out}$, as one might take arbitrary many `rounds' in the loop before exiting.
\color{black}
\end{example}

By definition, any \emph{finite} sequence of reactions that corresponds to the reaction labels of a walk in $\mGuf$ with start vertex $\ast_{in}$ and end vertex $\ast_{out}$ is a fast chain. 
Define the set of such walks by
\begin{align*}\mWu&:=\{(\gamma_1,\cdots , \gamma_l)\in(\mEuf)^l \st l\ge 2, (\gamma_1,\cdots , \gamma_l)\text{ is a walk } \\  & \qquad \text{  in }\mGuf \text{ with } 
 s(\gamma_1)= *_{in}, \ t(\gamma_l)=*_{out}\}.
\end{align*}
 Note that $\mWu$ might be infinite, even for finite RNs,  as in Example \ref{cont_ex} above.

We will consider the situation in which the reactions of $\mF$ occur  fast compared to the remaining reactions $\mR\setminus\mF$, in the sense that the transition intensities of the reactions in $\mF$ are scaled by $1/ \epsilon$ for small $\epsilon$. Thus, we consider a fast-slow dynamic regime. In that case, it is natural to expect that whenever a non-interacting molecule is created, then it will be degraded almost instantaneously through fast reactions, before any other non-fast reaction occurs. Such sequences of reactions (creation and degradation) are encoded in the walks of $\mWu$. To understand the fast dynamics, it is therefore important to understand the net gain of core species in the walks and their probabilities of occurring. 

In preparation for this, consider a walk
 $\Gamma=(\gamma_1,\gamma_2,\ldots,\gamma_l)\in \mWu$, 
and denote
$$w_i:=y_{L(\gamma_i)}+\sum_{j=1}^{i-1}y_{L(\gamma_j)}-y_{L(\gamma_j)}',\quad i=1,\ldots,l,$$
where $r_{L(\gamma_j)}$ is the reaction label, and  $L(\gamma_j)$ the reaction index, of the edge $\gamma_j$. Define
\begin{equation*}\label{red_net}
\mathrm{r}(\w):= \max(w_1,\ldots ,w_l).
\end{equation*}
Note that $\mathrm{r}(\w)$ depends on the order of the elements of $\Gamma$.

\begin{lemma}\label{only_core}
Let  $\w=(\gamma_1,\gamma_2,\ldots,\gamma_l)\in \mWu.$ 
Then the following holds:
\begin{itemize}
\item$\mathrm{r}(\w)\ge 0$, and $\mathrm{p}(\w):=\mathrm{r}(\w)+\sum_{i=1}^ly_{L(\gamma_i)}-y_{L(\gamma_i)}'\ge 0$ 
\item $p_\mU(\mathrm{r}(\w))=0$, $p_\mU(\mathrm{p}(\w))=0$.
\end{itemize}
\begin{proof}
We give complete proofs here for convenience, but note that the results also follow from \cite{hoessly2021sum}.
The second item follows by definition. As we take the maximum  coordinate-wise and $w_1\ge 0$, then $\mathrm{r}(\w)$ has non-negative coordinate in each species. 
To see that $\mathrm{p}(\w)$ has non-negative coordinate in each species, we note that by definition $y_{L(\gamma_l)}+\sum_{i=1}^{l-1}y_{L(\gamma_i)}-y_{L(\gamma_i)}'\leq \mathrm{r}(\w)$. Adding $\sum_{i=1}^ly_{L(\gamma_i)}'-y_{L(\gamma_i)}$ to both sides, we get $0\le y_{L(\gamma_l)}'\leq \mathrm{p}(\w)$, and the result follows.
\end{proof}
\end{lemma}

By Lemma \ref{only_core}, 
we might consider $\mathrm{r}(\w)$ and $\mathrm{p}(\w)$ as elements in $\Z_{\geq 0}^p$, and  $\mathrm{r}(\w)\to \mathrm{p}(\w)$ as the \emph{reduced reaction} obtained by contraction along $\Gamma$. Lemma \ref{non-neg_min} below justifies this view, and relates the compatibility requirement in Assumption \ref{ass1} to the reduced reactions.

\begin{example}\label{ex:MM3}
We continue with Example \ref{cont_ex}. Consider the walk $\w_0$ of $\mG_{\mU,\mF}$ consisting of the two edges associated to $r_1,r_2$. Then
 $$\mathrm{r}(\w_0)\to \mathrm{p}(\w_0)=E+A\to E+A,$$
a trivial reaction.
Likewise, the walk $\w^A_1$ consisting of the edges associated to $r_1,r_3,r_4,r_2$, also results in a trivial reaction,
$$\mathrm{r}(\w^A_1)\to \mathrm{p}(\w^B_1)=E+A+B\to E+A+B.$$
 In contrast, the walk $\w^B_1$ consisting of the three edges associated to $r_1,r_3$ and $r_5$ yields
$$\mathrm{r}(\w^B_1)\to \mathrm{p}(\w^A_1)=E+A+B\to E+A+Q.$$

Any other walk of $\mWu$ consists   either of a sequence of edges with labels  $r_1, r_3, r_4,\ldots,$ $r_3,r_4,r_3,r_5$ or a sequence of edges with labels $r_1, r_3, r_4,\ldots,r_3,r_4,r_2$. As the net gain of core species in the contraction of the reactions $r_3, r_4$ is zero, the corresponding reduced reactions are the same as that of $\w^A_1$, respectively, $\w^B_1$. Hence, there is only one non-trivial reduced reaction. For future reference, denote the corresponding walks with $n$ instances of $r_3$ by $\w^A_n$, respectively,  $\w^B_n$, for $n\ge 1$.
\end{example}

\begin{lemma}\label{non-neg_min}
Suppose Assumption \ref{ass1} holds. Let
$\w=(\gamma_1,\gamma_1,\ldots,\gamma_l)\in \mWu$. Furthermore, let $z\in\Z_{\geq 0}^p$ and  $x=(z,0)\in\Z^n_{\geq0}$. Then
$$\lambda_{L(\gamma_j)}\!\!\left(x+\sum_{i=1}^{j-1}y_{L(\gamma_i)}'-y_{L(\gamma_i)}\right)> 0,\quad j=1,\cdots, l-1,$$
 if and only if $z\geq \mathrm{r}(\w)$.
\end{lemma}

The lemma implies that the reactions corresponding to a walk in $\mWu$ can fire in succession of each other (without other non-fast reactions firing in between), if and only if the present  molecule counts of core species is larger or equal to $\mathrm{r}(\w)$. A similar statement appears in  \cite[Corollary 3.2]{hoessly2021sum}.

Lemma \ref{non-neg_min} allows us to define transition intensities of the reduced reactions in a natural way.  For $\w=(\gamma_1,\gamma_2,\ldots,\gamma_l)\in \mWu$, define the function $\Lambda_{\w}\colon\Z_{\geq 0}^p\to \R_{\geq0}$,
\begin{equation}\label{eq:bigLambda}
\Lambda_{\w}(z):=\lam_{L(\gamma_1)}(z,0)\prod_{j=2}^{l}\frac{\lam_{L(\gamma_j)}((z,0)+\sum_{i=1}^{j-1}\xi_{i})}{\sum_{\tilde{\gamma}\in \text{out}(s(\gamma_{j}))} \lam_{L(\tilde{\gamma})}((z,0)+\sum_{i=1}^{j-1}\xi_{i}) },
\end{equation}
where $\xi_{i}:=y_{L(\gamma_i)}'-y_{L(\gamma_i)}$, and $\text{out}(s(\gamma_{j}))$ denotes the set of outgoing edges of $s(\gamma_{j})$ in $\mGuf$.  Each term in the product is the probability that the desired reaction is chosen out of all possible fast reactions with the same non-interacting species in the reactant. The convention $0/0=0$ is used. Even if the RN is infinite, then \eqref{eq:bigLambda} is well-defined due to \eqref{eq:lambda}.

 \begin{definition}
Let $\mN=(\mC,\mR)$ be an RN on $\mS$,  let $\mU\subseteq\mS$  be a non-interacting set of species, and $\mF\subseteq\mR_\mU$ a set of proper fast reactions.
The \emph{reduced RN} $\mNuf=\mGr$ on $\mO=\mS\setminus\mU$, obtained by elimination of $(\mU,\mF)$ from $\mN$,   
is the possibly infinite RN defined by
$$ \mRfu=(\mR\setminus \mRu)\cup \{\mathrm{r}(\w)\to\mathrm{p}(\w)\st \w\in \mWu\},$$
and $\mC_{\mU,\mF}$ being the set of  vertices of $\mR_{\mU,\mF}$.
  \end{definition}

For $r\in\mR_{\mU,\mF}$,  define
 $$\mWu(r):=\{\w\in \mWu\st \mathrm{r}(\w)\to \mathrm{p}(\w)=r\}.$$

\begin{definition} \label{defredSnet}
Let $\mN=(\mC,\mR)$ be an SRN on $\mS$ with transition intensities  $\lam_r(\cdot),r\in \mR$, satisfying Assumption \ref{ass1}. Let $\mU\subseteq\mS$ be a set of non-interacting species, and $\mF\subseteq\mR_\mU$ a set of proper fast reactions.
Then, the possibly infinite SRN $\mNuf=\mGr$ on $\mO\setminus\mU$ with transition intensities $\lamm_r(\cdot), r\in \mRfu$, given by
$$\lamm_{r}(z):=\begin{cases}
\sum\limits_{\w\in \mWu(r)}
\Lambda_{\w}(z) & \text{ if $r\not\in \mR\setminus \mRu$},\\
  \lam_r(z,0)+\sum\limits_{\w\in \mWu(r)}
  \Lambda_{\w}(z) & \text{ if $r\in \mR\setminus \mRu$},\\
\lam_r(z,0) & \text{ if $r\in \mR\setminus \mRu$ and }\\
&\qquad r \not\in \{\mathrm{r}(\w)\to\mathrm{p}(\w)\st \w\in \mWu\},  
\end{cases}$$
is the \emph{reduced SRN}, obtained by elimination of $(\mU,\mF)$ from $\mN$. 
\end{definition}

As a reduced SRN has species set $\mO=\mS\setminus\mU$, the associated CTMC lives in $\Z_{\geq 0}^p$.
Even if the reduced SRN has infinitely many reactions, the associated CTMC is always well-defined, that is, it has no instantaneous states \cite{norris}. As remarked earlier, one might discard any trivial reaction. However, the transition intensities of the trivial reactions play a crucial role in Assumption \ref{ass2} in Section \ref{lim_stand}.

\begin{lemma}
Suppose Assumption \ref{ass1} holds. Then, the reduced SRN has  Q-matrix with all off-diagonal row sums finite.
\end{lemma}

\begin{proof}
By definition, we need to show that 
$$ \sum\limits_{r\in \mRfu}\lamm_r(z)<\infty.$$ 
For this, we note that
$$ \sum\limits_{r\in \mRfu}\lamm_r(z)=\sum\limits_{r\in \mR\setminus \mRu} \lam_r(z,0)+\sum\limits_{\w\in \mWu}
  \Lambda_{\w}(z).$$ 
Hence, it is enough to show that the second summand in the second term is finite as the first is finite by assumption. 
By construction of the reduced transition intensities,
$$\sum\limits_{\w\in \mWu}\Lambda_{\w}(z)\leq \sum\limits_{r\in \mRb\setminus \mRa}\lam_{r}(z,0)<\infty.$$
\end{proof}

We note that the last inequality is not necessarily an equality, see Example \ref{changedMM}.
 As a consequence of Lemma \ref{non-neg_min}, Assumption \ref{ass1} holds for a reduced SRN, provided it holds for the original SRN. 

\begin{corollary}\label{comp_cor}
Suppose Assumption \ref{ass1} holds. Then, for any reaction $r:y\to y\p\in \wRn$ of the reduced SRN, it holds that $\lamm_{r}(z)> 0$ if and only if $z\geq y$.
\end{corollary}

\begin{proof}
It is enough to prove it for $\Lambda_{\w}(z),$ where $\w\in \mWu$ and $z\in\Z_{\geq 0}^p$, that is,
 $$\Lambda_{\w}(z)>0\iff
z\geq \mathrm{r}(\w).$$
Only one direction is not obvious. Assume $z\geq\mathrm{r}(\w)$. Using the definition of $\Lambda_{\w}(\cdot)$, it is sufficient to show that the numerators in the fraction are all non-zero. This holds by Lemma \ref{non-neg_min}, hence  $\Lambda_{\w}(z)>0$.
\end{proof}

\begin{example}\label{calc_ex_kin}
We continue with Examples \ref{cont_ex}, \ref{ex:MM3} assuming stochastic mass-action kinetics. There are three reduced reactions with the first two being trivial,
$$s_1\colon E+A \to E+A,\quad s_2\colon E+A+B\to E+A+B,\quad s_3\colon E+A+B\to E+A+Q,$$
with infinitely many walks underlying the second and third reduced reaction. Here, $s_i$ is used to denote the  reactions of the reduced SRN, rather than $r_i$, to distinguish the reactions of the reduced RN from those of the original RN. We find 
$$\Lambda_{\w_0}(z)=\frac{\k_1\k_2z_Ez_A}{k_3z_B+\k_2},$$
$$\Lambda_{\w^A_n}(z)= \frac{\k_1\k_2z_Ez_A}{k_3z_B+\k_2}  \left(\frac{\k_3\k_4z_B}{(\k_4+\k_5)(k_3z_B+\k_2)}\right)^{\!\!n},\quad n\ge 1,$$
$$\Lambda_{\w^B_n}(z)= \frac{\k_1\k_3\k_5z_Ez_Az_B}{(\k_3z_B+\k_2)(\k_4+\k_5)}  \left(\frac{\k_3\k_4z_B}{(\k_4+\k_5)(k_3z_B+\k_2)}\right)^{\!\!n-1},\quad n\ge 1,$$
such that
$$\lamm_3(z)=\sum_{n=1}^\infty\Lambda_{\w_n^B}(z)=
\frac{\k_1\k_3\k_5z_Ez_Az_B}{(\k_4+\k_5)(\k_3z_B+\k_2)-\k_3\k_4z_B}.$$
Analogously, we  compute the sums over walks for the trivial reactions. 
We have $\lamm_1(z)=\Lambda_{\w_0}(z)$ and
 $$\lamm_2(z)=
\frac{\k_1\k_3\k_5z_Ez_Az_B}{(k_3z_B+\k_2)((\k_4+\k_5)(\k_3z_B+\k_2)-\k_3\k_4z_B)}.$$
Furthermore, a simple calculation gives the following identity for $z\in \Z_{\geq 0}^p$,
$$ \sum\limits_{\w\in \mWu}\Lambda_{\w}(z)=\lamm_1(z)+\lamm_2(z)+\lamm_{1}(z)=\lam_1(z,0).$$
An interesting observation is that the kinetics is \emph{mass-action-like} in the sense that the numerator alone is of mass-action form, while the denominator is positive (for any state). This is, however, not true in general, see Example \ref{changedMM}.
\end{example}

\begin{example}\label{ex_inf1}
Consider the RN, 
\[r_1\colon S_1\ce{->}U_1,\quad r_2\colon U_1\ce{->}S_3,\quad r_3\colon S_2+U_1\ce{->}U_1,\]
with  $\mU=\{U_1\}$ and $\mRf=\mRa$  proper. Taken with mass-action kinetics with corresponding  rate constants, $\k_1,\k_2,\k_3$, the reduced SRN has infinitely many reactions, $s_n\colon S_1+nS_2\to S_3$, $n\ge0 $, 
with transition intensities, 
$$\lamm_n(x)=\frac{\k_1\k_2x_{S_1}}{\k_2+\k_3({x_{S_2}-n})}
\prod_{i=0}^{n-1}\frac{\k_3({x_{S_2}-i})}{\k_2+\k_3({x_{S_2}-i})},\quad n\ge 0. $$
On any particular state, at most finitely many reactions can be \emph{active}, that is, have non-zero transition intensity.
\end{example}

For  examples from the biochemical literature, we refer to $\S$ \ref{sec_zoo}.

\section{Two-scale SRNs}
\label{reduced_dynamics}

We study transient approximability of the CTMC via the reduced SRNs under 
appropriate assumptions. We cover two settings, the standard time-homogeneous CTMC setting for SRNs, and afterwards the setting where the transition intensities of the SRN are allowed to be time-dependent.

\subsection{Transient approximation}\label{lim_stand}

Let $\mN=(\mC,\mR)$ be an SRN on a species set $\mS$ (of size $n$) with   non-interacting species  $\mU\subseteq \mS$,  fast proper reactions $\mRf\subseteq\mR_\mU$, and transition intensities $\lambda_r, r\in\mR$.
From this SRN, we construct a family of SNRs, indexed by a parameter $\epsilon>0$, with corresponding $Q$-matrix $Q^\epsilon$. In particular, the transition intensities of the reactions in $\mRf$ consuming non-interacting species are  scaled in $\eps$ as follows: 
 \begin{itemize}
 \item $\lam^\eps_{r}(x):=\frac{1}{\eps}\lam_{r}(x)$, $r\in\mF$, that is, for small $\epsilon$  the reactions are fast compared to the remaining reactions,
\item  $\lam^\eps_{r}(x):=\lam_{r}(x)$, $r\in\mR\setminus \mRf$, that is, the transition intensities are independent of $\epsilon$, and the reactions are slow. 
\end{itemize} 
Thus, we might write the corresponding $Q$-matrix as as sum of two terms, a fast part $\widetilde{Q}$,  scaled by $\frac{1}{\eps}$, and a slow part $\widehat{Q}$,
 $$Q^\eps=\frac{1}{\eps}\widetilde{Q}+\widehat{Q}.$$
It follows from Definition \ref{defredSnet}  that the reduced SRN, obtained from the SRN with transition intensities $\lam^\eps_{r}$, $r\in\mR$, has transition intensities $\tau_r$, $r\in \mR_{\mU,\mF}$, independent of $\eps$.

Let $(X_\eps(t))_{t\ge 0}$ be Markov chains on $\Z_{\ge0}^n$ with generators $Q^\eps$, $\eps>0$, respectively.  We note that all $Q^\eps$, $\eps>0$, are dynamically equivalent in the sense that the connectivity of the state space and the decomposition of the state space $\Z^n_{\ge0}$ into communicating classes   are independent of $\eps$. Furthermore, let $(X_0(t))_{t\ge0}$ be a Markov chain on $\Z_{\ge0}^p$ (excluding the $m=n-p$  coordinates for non-interacting species) with generator $Q^0$, obtained from the transition intensities $\tau_r$, $r\in \mR_{\mU,\mF}$.
 In the following, we provide conditions that guarantee that the dynamics of $(X_\eps(t))_{t\ge 0}$ is \emph{similar} (in a sense to be made precise) to the dynamics of $(X_0(t))_{t\ge0}$, whenever $\eps$ is small. Technically, we will consider the limit as $\eps\to0$.

We restrict  attention to a particular \emph{closed} set $E$ of $Q^\eps$, such that $E\cap (\Z_{\geq 0}^p\times \{0\})\not=\emptyset$. That is, there are states in $E$ with no molecules of non-interaction species being present. If the reduced SRN starts in $X_0(0)\in E_0:=\rho_\mO(E\cap (\Z_{\geq 0}^p\times \{0\}))$, then by construction of the reduced reactions, $X_0(t)\in E_0$ for all $t>0$. Furthermore, $E_0$ is a closed set of $Q^0$.

We require the following. 
\begin{assumption}
\label{ass2}
The closed set $E$ is finite, and   all 
$z\in E_0$ satisfy the following:
 \begin{equation}\label{eq:fin_out3}\
 \sum\limits_{\w\in \mWu}\Lambda_{\w}(z)= \sum\limits_{r\in \mRb\setminus \mRa}\lam_{r}(z,0)
\end{equation}
\end{assumption}

The sum on the left hand side  also includes walks giving rise to trivial reactions, as  in Example \ref{calc_ex_kin}. The  condition might fail in two ways. Either a walk is \emph{blocked} because of lack of molecules of core species (for example $r_1\colon0\to U,$ $r_2\colon S+U\to 0$; if there no molecules of $S$, then the fast chain $r_1,r_2$ is blocked and a molecule of $U$ cannot be degraded), or one might be \emph{trapped} in infinite walks (for example, $r_1\colon0\to U$, $r_2\colon2S+U\to 3S+U$, $r_3\colon U\to 0$; there is positive probability of an infinite walk $r_1,r_2,r_2,\ldots$ without degradation of the $U$ molecule). The latter cannot occur if $E$, and hence $E_0$, are finite as this implies that at most finitely many (reduced) reactions can be active on any state of $E$ ($E_0$). 

For a sub-conservative RN, any closed set of states $E$ (not necessarily a communicating class) in a directed stoichiometric compatibility class is finite.
If Assumption \ref{ass1} is satisfied, then the conditions in Assumption \ref{ass2} do not depend on the transition intensities $\lam_r$, $r\in\mR$,  but  only on the structure of the underlying RN  (together with $\mU$ and $\mRf$).  Furthermore, as \eqref{eq:lambda} holds by assumption, then Assumption \ref{ass2} is meaningful even for infinite RNs. In particular, Theorem \ref{main_thm_stand} also holds for infinite RNs as $E$ is finite. 

For a subset $B\subseteq E\cap (\Z_{\geq 0}^p\times \{0\})$,  we define  $B_0=\rho_\mO(B)$.

\begin{theorem}\label{main_thm_stand}
Let $\mN=(\mC,\mR)$ be an SRN on $\mS$ with transition intensities $\lambda_r$, $r\in\mR$, $\mU\subseteq \mS$  a set of non-interacting species,  and $\mRf\subseteq\mR_\mU$ a set of  fast proper reactions.
 Suppose Assumptions \ref{ass1} and \ref{ass2} hold.
 Let $\pi$ be a probability distribution on $E\cap (\Z_{\geq 0}^p\times \{0\})$, and $\pi_0$ the induced probability distribution on $E_0$ by omitting the last $n-p$ coordinates of the states of $E$.
 Then, the following holds for any $0<T$:
 $$\sup_{t\in[0,T]}|P_{\pi}(X_\eps(t)\in B)-P_{\pi_0}(X_0(t)\in B_0)|=O(\epsilon)\quad \text{for }\eps\to 0.$$
In particular, for any $B\subseteq E\cap (\Z_{\geq 0}^p\times \{0\})$ and any $0<T$:
$$\lim_{\eps\to 0}\sup_{t\in[0,T]}|P_\pi(X_\eps(t)\in B)-P_{\pi_0}(X_0(t)\in B_0)|=0.$$
 \end{theorem}

The proof is in  Appendix $\S$ \ref{sect_proof_stand}. As a consequence of the theorem we have:

 \begin{corollary}\label{cor:lim}
Assume as in Theorem \ref{main_thm_stand}.
 Let $x=(z,0)\in E\cap (\Z_{\geq 0}^p\times \{0\})$ and $B\subseteq E\cap (\Z_{\geq 0}^p\times \{0\})$. Then, for all $t\geq0$ it holds that:
 $$\lim_{\eps\to 0}P_{x}(X_\eps(t)\in B)=P_{z}(X_0(t)\in B_0).$$
 \end{corollary}
 
If the initial state $X_\eps(0)$ has more than one molecule of the non-interacting species, then these will in general be depleted quickly for small $\eps$. This is however not always the case, see   Example \ref{changedMM}.

\subsection{The case of time-dependent transition intensities}\label{lim_stand2}

In the previous section, we considered the scaling limit  of time-homogeneous SRNs. While time-constant transition intensities are reasonable in many situations, time-dependence becomes relevant, for example, in connection with variation in experimental setups, cycle-dependent mechanisms or temperature changes \cite{anderson_time_dep}. 

The setup is essentially the same as in Section \ref{lim_stand}, except we now allow time-dependent  transition intensities, $\lam_r(t,x)$, $r\in\mR$, for $t$ in a compact time interval $ t\in [0,T]$.
We make the following assumption.

\begin{assumption}\label{ass4}
The time-dependent transition intensities, $\lam_r(t,x)$, $r\in\mR$, are such that $\lam_r(t,\cdot)$ satisfies Assumption \ref{ass1} for all $t\in [0,T]$ and such that for all $x\in \Z^n_{\geq0}$, $\lam_r(\cdot,x)$, $r\in\mR$, are $C^1$-functions from $[0,T]\to \R_{\geq 0}$  with Lipschitz derivatives. 
\end{assumption}

The  transition intensities are scaled as in the homogeneous case.
Under Assumption \ref{ass4}, the reachability properties of the CTMC are invariant over time and match the reachability properties of the homogeneous SRNs under Assumption \ref{ass1}. Assumption \ref{ass2} holds for all $t\in[0,T]$, provided it holds for one $t$, as the assumption is structural.

With these changes and amendments, then Theorem \ref{main_thm_stand} and Corollary \ref{cor:lim} hold as well for the time-dependent case. A proof is sketched in Appendix \ref{proof_nonhom}.

\section{Examples and properties}\label{sectfi}

We next illustrate the reduction procedure with various examples. Furthermore, we  derive  sufficient conditions for Assumption \ref{ass2} to hold, and   compare the long-term behavior of the CTMC of the reduced SRN with the one of the original SRN. 

\subsection{Zoo of examples}\label{sec_zoo}

In this section, we give realistic examples to elaborate on the reduction. All examples are taken with stochastic mass-action kinetics, hence they satisfy Assumption \ref{ass1}. To indicate mass-action kinetics, we put the reaction rate constants $\k_i$ as labels of the reactions, and omit the reaction names $r_i$. A key source book for realistic examples is  \cite{syst_ing}, that contains reaction networks used in systems biology. Other useful references are \cite{bowden}, which focuses on enzyme kinetics, and \cite{Murray2002}, which has examples from a range of areas in biology.

\begin{example}\label{red_2MM}
We first consider the two-substrate Michaelis-Menten mechanism \cite[Section 3.1.2]{syst_ing}; also discussed in Examples \ref{two_sub_Ing}, \ref{cont_ex}, \ref{ex:MM3}, and \ref{calc_ex_kin}, and repeated here for convenience with reaction names replaced by their reaction constants:
$$E+A\ce{<=>[\k_1][\k_2]}EA,\quad EA+B\ce{->[\k_3]}EAB
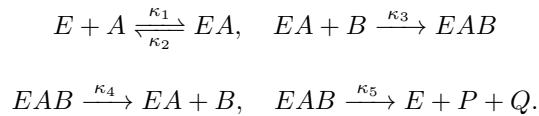
$$
\vspace{-.2cm}
$$ EAB\ce{->[\k_4]} EA+B,\quad EAB\ce{->[\k_5]}E+P+Q.$$
  The mechanism is commonly studied as a deterministic mass-action system, using Tikhonov-Fenichel singular perturbation theory or the QSSA, with short-lived \emph{intermediate complexes} (non-interacting species) $\mU=\{EA,EAB\}$ and fast reactions $\mF=\mR_\mU=\{r_2,r_3,r_4,r_5\}$. Assuming the same in the stochastic setting gives the reduced RN in Example \ref{ex:MM3} with one non-trivial reaction,
$$ s_1\colon E+A+B\to E+P+Q,$$
with transition intensity
$$\lamm_{1}(z)=\frac{\k_1\k_3\k_5z_Ez_Az_B}{(\k_4+\k_5)(\k_3z_B+\k_2)-\k_3\k_4z_B}.$$
Assumption \ref{ass2} is fulfilled for all $z\in \Z_{\geq 0}^5$  ($p=5$, see Example \ref{calc_ex_kin}), in particular the original RN is 	conservative. Then, Theorem \ref{main_thm_stand} applies to any closed set $E$ in a directed stoichiometric compatibility class that intersects $ \Z_{\geq 0}^5\times \{0\}$ non-trivially. 
\end{example}

\begin{example}\label{changedMM}
Consider the two-substrate Michaelis-Menten mechanism  with  $\mU=\{EA,EAB\}$ as above, but fast reactions $\mF=\{r_3,r_4,r_5\}\subseteq\mR_\mU$, that is, we remove $r_2$ from the fast reactions in the previous example. Then there is only one reduced reaction in contrast to Example \ref{red_2MM} that additionally had two trivial reactions, given by
$$ s_1\colon E+A+B\to E+P+Q,$$
with transition intensity
\begin{equation*}\label{identity}\lamm_{1}(z)=\k_1z_Ez_A\1_{\{z_B\colon z_B\geq 1\}}.\end{equation*}
Assumption \ref{ass2} holds for all $z=(z_A,z_B,z_E,z_P,z_Q)\in\Z_{\ge 0}^5$ 
with $z_B>0$:
$$ \sum\limits_{\w\in \mWu}\Lambda_{\w}(z)=\lamm_{1}(z)=\k_1z_Ez_A\1_{\{z_B\colon z_B\geq 1\}}=\k_1z_Ez_A=\lam_1(z,0).$$
Note, that this example does not have mass-action-like kinetics,  as $\tau_1(z)$ cannot be expressed as a fraction with positive denominator such that the numerator is of mass-action form. Alternatively, we might extend by multiplication and division by $z_B$.
\end{example}

\begin{example}\label{ex_13}
Consider a non-competitive inhibition network \cite[Section 3.2.2]{syst_ing}:
$$S+E\ce{<=>[\kappa_1][\kappa_2]}ES \ce{->[\kappa_3]} E+P,\quad I+E\ce{<=>[\kappa_4][\kappa_5]}EI, $$
\vspace{-.3cm}
$$I+ES\ce{<=>[\kappa_6][\kappa_7]}ESI, \quad S+EI \ce{<=>[\kappa_8][\kappa_9]}ESI,$$
where $E$ is an enzyme that catalyze the conversion of a substrate $S$ into a substrate $P$. The conversion is delayed by an inhibitor $I$ that binds to the enzyme and to the substrate through  the intermediate complex $EI$. It is non-competitive in the sense that $I$ does not compete with $E$ for substrate binding, but rather acts on $E$ directly; a phenomenon known as allosteric regulation.

This example is typically analyzed by means of the QSSA with short-lived species; here $\mU=\{ES,EI,ESI\}$. In our setting, taking  $\mU$ to be non-interacting species with  $\mRf=\mRa$, then the reduced SRN becomes
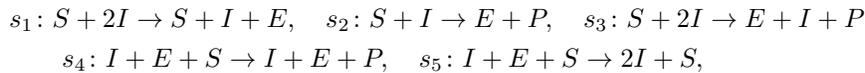
\begin{align*}
& s_1\colon S+2I\to S+I+E,\quad s_2\colon S+I\to E+P,\quad s_3\colon S+2I\to E+I+P\\
& \qquad s_4\colon I+E+S\to I+E+P, \quad s_5\colon I+E+S\to 2I+S,
 \end{align*}
 with  transition intensities

\smallskip
 \footnotesize
 \begin{align*}
\lamm_{1}(z)&=\frac{\k_1\k_5\k_6\k_9z_Sz_I(z_I-1)}{(\k_2+\k_3)(\k_7(\k_5+\k_8z_S)+\k_5\k_9)+\k_6\k_5\k_9(z_I-1)}\\
\lamm_{2}(z)&=\frac{\k_1\k_3z_Sz_I}{\k_2+\k_3+\k_6(z_I-1)}\\
\lamm_{3}(z)&=\k_1\k_3z_Sz_I\left(\frac{\k_7(\k_5+\k_8z_S)+\k_5\k_9}{(\k_2+\k_3)(\k_7(\k_5+\k_8z_S)+\k_5\k_9)+\k_6\k_5\k_9(z_I-1)}-\frac{1}{\k_2+\k_3+\k_6(z_I-1)}\right)\\
\lamm_{4}(z)&=\frac{\k_3\k_4\k_7\k_8(\k_7+\k_9)z_Ez_Iz_S}{[\k_5(\k_7+\k_9)+\k_8\k_7z_S]((\k_2+\k_3)(\k_7+\k_9)+\k_6\k_9z_I) )-\k_8\k_7\k_6\k_9z_Iz_S}\\
\lamm_{5}(z)&=\frac{\k_2\k_4\k_7\k_8(\k_7+\k_9)z_Ez_Iz_S}{[\k_5(\k_7+\k_9)+\k_8\k_7z_S]((\k_2+\k_3)(\k_7+\k_9)+\k_6\k_9z_I) )-\k_8\k_7\k_6\k_9z_Iz_S}.
  \end{align*}
 \normalsize

\medskip
For any closed set $E$ in a directed stoichiometric compatibility class that intersects $ \Z_{\geq 0}^4\times \{0\}$ non-trivially, Assumption \ref{ass2} is satisfied.
\end{example}

\begin{example}\label{allost_act}
Consider an example of allosteric activation \cite[Problem 3.7.8]{syst_ing} that models the regulation of an enzyme by binding of an allosteric activator before the enzyme can bind a substrate:
$$
R+E\ce{<=>[\kappa_1][\kappa_2]}ER,\quad ER+S \ce{<=>[\kappa_3][\kappa_4]} ERS\ce{->[\kappa_5]}P+ER.$$
Here, $R$ is the allosteric activator, $S$ the substrate, $E$ the enzyme, and $P$ the product. While the RN is reminiscent of the the two-substrate example \ref{red_2MM}, the enzyme-activator complex $ER$ stays intact after the product $P$ dissociates in  reaction $r_5$.

To analyze the behavior of the system, often the QSSA is applied to a  mass-action ODE system with short-lived intermediate complexes (non-interacting species) $\mU=\{ER,ERS\}$.
Choosing analogously $\mRf=\mRa$ in the stochastic case gives a proper  set of fast reactions. Then, the reduced RN has infinitely many reactions, given by:
\begin{equation*}
s_{2k-1}\colon E+R+kS\to E+R+kP,\quad  s_{2k}\colon E+R+(k+1)S\to E+R+kP+S,
 \end{equation*}
for $k\ge 1$, with
transition intensities 
$$\lamm_{2k-1}(z)= \frac{\k_1\k_2z_Rz_E}{\k_2+\k_3(z_S-k-2)}\prod_{i=0}^{k-1}\frac{\k_3\k_5(z_S-i)}{\k_2\k_4+\k_2\k_5+\k_3\k_5(z_S-i)}, $$
$$\lamm_{2k}(z)=\lamm_{2k-1}(z)\frac{\k_3\k_4(z_S-k)}{\k_2\k_4+\k_2\k_5+\k_3\k_5(z_S-k)}.$$
All $z\in \Z_{\geq 0}^4$ ($p=4$) satisfy  \eqref{eq:fin_out3}. As $\mN$ is sub-conservative, any closed set $E$ of a directed stoichiometric compatibility class is finite. Hence, for any closed set $E$ in a directed stoichiometric compatibility class that intersects $ \Z_{\geq 0}^4\times \{0\}$ non-trivially, Assumption \ref{ass2} is satisfied. 
\end{example}

\begin{example}\label{suicide_substr}
As a final example, consider  a mechanism-based inhibitor system \cite[Section 6.4]{Murray2002},
\begin{align*}
S+E\ce{<=>[\kappa_1][\kappa_2]}X\ce{->[\kappa_3]} & Y\ce{->[\kappa_6]}E+P, \quad
Y\ce{->[\kappa_4]} E_i,
\end{align*}
 where $S,P$ are  substrate and product, respectively, $X,Y$ intermediate complexes, $E$ an enzyme in active form (implying it might bind to the substrate) and $E_i$, the enzyme in its inactivated form. A main interest is to know the final ratio of the product to the inactivated enzyme  \cite[Section 6.4]{Murray2002}. For this, often the QSSA is applied to the short-lived species $\mU=\{X,Y\}$. Analogously, we consider the stochastic system with $\mF=\mRa$ being fast and proper, and study the corresponding stochastic reduction. 
 The reduced SRN has reactions
$$  s_1\colon S+E\to E+P,\quad s_2\colon S+E\to E_i$$
with  transition intensities:
$$\lamm_{1}(z)=\frac{\k_1\k_3\k_5z_Sz_E}{(\k_2+\k_3)(\k_4+\k_5)} , \quad \lamm_{2}(z)=\frac{\k_1\k_3\k_4z_Sz_E}{(\k_2+\k_3)(\k_4+\k_5)}.$$
As  $\mU=\{X,Y\}$ are intermediate species, all $z\in \Z_{\geq 0}^4$ ($p=4$) satisfy  \eqref{eq:fin_out3}. As $\mN$ is sub-conservative, for any closed sets $E$ of a directed stoichiometric compatibility class Assumption \ref{ass2} is satisfied. 
\end{example}

\subsection{Properties enabling sufficient conditions for simplification}\label{secondl}

While Assumption \ref{ass1} and $\mF$ being proper are both easy to check, Assumption \ref{ass2} is non-trivial in general.
Assumption \ref{ass2} requires a finite closed set of the CTMC  and that molecules of non-interacting species can be degraded through chains of fast reactions. The following gives sufficient conditions for this to hold (see  Appendix $\S$ \ref{additional_vers} for proof).

\begin{proposition}\label{suff_cond_fin_alt} Assume $\mN$ is a sub-conservative SRN on a species set $\mS$,  $\mU\subseteq\mS$ is a set of non-interacting species, and $\mF\subseteq\mR_\mU$ is a proper set of reactions. Suppose Assumption \ref{ass1} holds. Furthermore, assume that one of the following holds:
\begin{itemize}
\item[(a)] $\mRf\cup (\mRb\setminus\mRa)$ is weakly reversible.
\item[(b)] $\mRf\cap \mRa\cap\mRb$ is weakly reversible, and for any reaction $r\colon y\to y'$ in  $\in \mRb\setminus\mRa$, the reverse reaction $r'\colon y'\to y$ is in $\mRf$.
\item[(c)] $\mN$ is weakly reversible and $\mF=\mRa$.
\end{itemize}
Then, for an arbitrary closed set  in a (directed) stoichiometric compatibility class, Assumption \ref{ass2} is satisfied.
\end{proposition}

None of the conditions in the theorem are necessary for Assumption \ref{ass2} to hold. For intermediate species a stronger result holds.

\begin{lemma}\label{interm_finite_fine_alt}
Assume $\mN$ is an SRN on a species set $\mS$, and that $\mU\subseteq\mS$ is a set of intermediate species. Suppose Assumption \ref{ass1} holds. Then, $\mF\subseteq\mR_\mU$  is proper if and only if all $z\in \Z_{\geq 0}^p$ satisfy  \eqref{eq:fin_out3} of Assumption \ref{ass2}.
\end{lemma}

We summarize the properties of the examples from $\S$ \ref{sec_zoo}. 
In examples    \ref{red_2MM},  \ref{ex_13} and  \ref{allost_act} we can  directly conclude by Proposition \ref{suff_cond_fin_alt} (b) that Assumption \ref{ass2} holds.

\medskip
\begin{center}
\setlength{\tabcolsep}{0.5em} 
{\renewcommand{\arraystretch}{1.2}
\begin{tabular}{ |c|c|c|c|c| } 
 \hline
Example   &  $\mN$ sub-cons& 
\eqref{eq:fin_out3} fulfilled & $\mNuf$ finite\\

 \hline
 \ref{ex_inf1}   &$\checkmark$&$\checkmark$&$-$\\
 \ref{red_2MM}    &$\checkmark$&$\checkmark$&$\checkmark$\\

  \ref{changedMM}    &$\checkmark$&$-$&$\checkmark$\\
 \ref{ex_13} &$\checkmark$&$\checkmark$&$\checkmark$\\
 \ref{allost_act} &$\checkmark$&$\checkmark$&$-$\\
 \ref{suicide_substr} &$\checkmark$&$\checkmark$&$\checkmark$\\

\hline

\end{tabular}}
\end{center}

\subsection{Comparison of the original RN and the reduced RN}

A natural question following the transient approximability of SRNs with non-interacting species is whether the reduced SRN approximates the stationary behavior in the limit as $\eps\to 0$. First one might ask whether states $x=(z,0)\in \Z_{\geq 0}^p\times \{0\}$ and $z$ are of the same type (positive recurrent/null recurrent/transient) for $\mN$ and $\mNuf$, respectively. Furthermore, the  stationary distribution (if it exists) for the scaled SRN $\mN$ as $\eps\to 0$ might match the stationary distribution  of $\mNuf$.  The following examples show that such correspondences do not hold in general, even for reduction by intermediates.

\begin{example} Consider the following SRN $\mN$ with  mass-action kinetics, 
$$S_1\ce{<=>[\kappa_1][\kappa_2]}U_1\ce{->[\kappa_3]} S_2\ce{->[\kappa_4]}S_1,\quad U_1\ce{->[\kappa_5]} S_3.$$
Choosing $\mU=\{U_1\}$, $\mRf=\{r_3\}$, then the reduced SRN $\mNuf$ is
$S_1\ce{<=>[][]}S_2 $ with transition intensities satisfying Assumption \ref{ass1} by Corollary \ref{comp_cor}.
Any $(a,b,c,0)$ with $a+b\geq 1$ is transient for $\mN$, but recurrent for $\mNuf$. Reaction $r_5$ in $\mN$ causes absorption into $(0,0,a+b+c+d,0)$. As $r_5$ is not present in the reduced SRN, this cannot happen in $\mNuf$.
\end{example}

\begin{example}
Consider the SRN $\mN$ with  mass-action kinetics,
$$S_1\ce{<=>[\kappa_1][\kappa_2]}U_1\ce{<=>[\kappa_3][\kappa_4]} S_2$$
Choosing $\mU=\{U_1\}$, $\mRf=\{r_3\}$, then the reduced SRN $\mNuf$  becomes
$S_1\to S_2 $.
The state $(a,b,0)$ with $a\geq 1$ is recurrent for $\mN$, but transient for $\mNuf$. 
\end{example}

The next result follows from \cite[Theorem 5.6]{hoessly2021sum}. We recall that irreducible components of sub-conservative RNs are finite \cite{book_kurtz_crns}. For an irreducible component $E\subseteq\Z^n_{\ge 0}$, define $E_0=\rho_{\mO}(E\cap (\Z_{\geq 0}^p\times \{0\}))$.

\begin{theorem}\label{thm_int}
Let $\mN=(\mC,\mR)$ be a sub-conservative SRN on $\mS$, $\mU\subseteq \mS$  a set of intermediate species,  and $\mRf=\mR_\mU$ a set of  proper reactions.
 Suppose Assumption \ref{ass1} holds. Then $x=(z,0)\in \Z_{\geq 0}^p\times \{0\}$ is transient for $\mN$ if and only if $z$ is transient for $\mNuf$, and the same holds for positive recurrence. 
 In particular, if a subset $E$ is an irreducible component for $\mN$ then $E_0$ is an irreducible component for $\mNuf$.
\end{theorem}

From Theorem \ref{thm_int}, Lemma \ref{interm_finite_fine_alt} and \cite[Theorem 3]{jia_fin}, we get convergence of the stationary distribution.

\begin{corollary}
Let $\mN=(\mC,\mR)$ be a sub-conservative SRN on $\mS$, $\mU\subseteq \mS$  a set of intermediate species,  and $\mRf=\mR_\mU$ a set of fast proper reactions.
 Suppose Assumption \ref{ass1} holds. 
Let $E$ be an irreducible component, and let $E_0$ be the corresponding projected set. 
Furthermore, denote by $\pi_\eps$ the unique stationary distribution of $Q_\eps$ on $E$, and let $\pi_0$ be the unique stationary distribution of the reduced SRN on $E_0$. Then, 
$\pi_\eps((z,0))\to \pi_0(z)$,  for $z\in E_0,$ and $\eps\to 0.$
\end{corollary}

\section{Relation to previous work and discussion}\label{disc}

We compare our setting for reduction of SRNs with non-interacting species to other approaches of model simplification.
Then we discuss assumptions, extensions and examples that are not covered by our approach.

In our treatment, we consider  SRNs with fast reactions consuming non-interacting species, generalizing earlier settings  in three ways: 1) extending from intermediate species to non-interacting species, 2) allowing arbitrary transition intensities satisfying a compatibility condition, and 3) allowing non-homogeneous Markov dynamics, that is, time-dependent transition intensities. As the set-up is based on \cite{yin_zhang_trans,yin2012continuous}, we loose the possibility to study multiple time scales,  but   are restricted to two scales.

Reductions of SRNs with non-interacting species resembles reduction by intermediate species  as intermediate species form a special class of non-interacting species \cite{CW16}.  In contrast to our setting, \cite{CW16} allows a multi-scale setting. Other  multi-scale reductions require the parameters to fulfill a \emph{balancing equation} between scales of species concentrations and scales of transition intensities \cite{kurtz2,approx_kurtz,BKPR06,KK13,Schnoerr_2017}.  This is not required in our setting (nor in \cite{CW16} for the  reactions involving intermediate species), in fact such equation will not hold. Balancing is violated because  fast reactions do not have  `limits'  themselves as $\eps\to 0$, but only  jointly through contraction of sequences of reactions. 
Furthermore, at the process-level, we do not have  convergence in  the Skorohod topology, see \cite[Example 5.3]{CW16} or \cite[$\S$ 6.5]{KK13}, because even for small $\epsilon$, molecules of non-interacting species are created and exist for small amounts of time, while this is not possible in the reduced SRN. Convergence in Skorohod topology typically require that species abundance is measured in concentrations, rather than in molecule numbers.

The results on the transient dynamics in our two-scale setting could potentially be extended by allowing 
general infinite state space (that is, weaken  Assumption \ref{ass2}). In so,  the transition rates of the original CTMC become unbounded for mass-action kinetics, and the theory on singular perturbations for CTMCs are not applicable \cite{yin_zhang_trans,yin2012continuous}. Furthermore, the original process might have sample paths that diverge to infinity in a finite amount of time.  Hence,  
 theory applicable to unbounded cases as well as conditions ensuring non-explositivity would be of interest to develop.  We note that even 1-dimensional CTMCs arising from RNs can have both positive recurrent and explosive irreducible components \cite{xu2019dynamics}, implying that different parts of the state space might have to be analyzed separately.  

  Another generalization would be to allow many scales rather than two scales. This could be done using \cite{jia_fin} in the finite state space case. However, it becomes difficult to find the  reduced reactions, as these depend on the concrete form of the transition intensities and their scalings. Also, a reduction might not be  interpreted as SRNs fulfilling Assumption \ref{ass1}.  As an example, consider
 $$
S_1\ce{->[\k_1]}U_1\ce{->[\frac{1}{\eps}\k_2]}S_2,\quad U_1+S_3\ce{->[\frac{1}{\eps^2}\k_3]}S_4
 $$
 with stochastic mass-action kinetics.
The reduction for small $\eps$ can be written as two reactions
 $$S_1\ce{->[]}S_2,\quad S_1+S_3\ce{->[]}S_4
 $$
 with transition intensities $\k_1z_{S_1}\1_{\{0\}}(z_{S_3})$ for the first reaction and $\k_1z_{S_1}\1_{\N\setminus\{0\}}(z_{S_3})$ for the second. In particular, the first transition intensity invalidates Assumption \ref{ass1}.

In our approach, it is important that non-interacting species are produced and degraded.
 Consider the SRN with reactions
$$U_1\ce{<=>[\kappa_1][\kappa_2]}U_2,\quad S_1+U_1\ce{<=>[\kappa_3][\kappa_4]}S_2+U_2.$$
Neither $U_1$ nor $U_2$ are produced nor degraded, hence there does not exist a proper set of reactions $\mF$ and the example falls outside our setting. As a matter of fact, the scaling parameter $\eps$ would apply uniformly to all reactions as they all transform one non-interacting species into another. Hence, the distribution of the CTMC approaches the stationary distribution within a short time span (for small $\eps$). This distribution is not concentrated on the part of the state space without molecules of non-interacting species. Rescaling time by $\eps$ would retrieve the original chain.

\section*{Acknowledgement}
LH acknowledges funding from the Swiss National Science Foundations Early Postdoc.Mobility grant (P2FRP2\_188023). The work presented in this article is supported by Novo Nordisk Foundation, grant NNF19OC0058354.

\appendix
\section{}
\label{additional_vers}

The following lemma  proves Proposition \ref{suff_cond_fin_alt} on sufficient conditions for Assumption \ref{ass2}.

\begin{lemma}
Assume $\mN=(\mC,\mR)$ is a sub-conservative SRN on a species set $\mS$ with transition intensities $\lambda_r$, $r\in\mR$, that  $\mU\subseteq\mS$ is a set of non-interacting species, and $\mF\subseteq\mR_\mU$ is a proper set of fast reactions.  Suppose Assumption \ref{ass1} holds.
Then either of the following are sufficient for Assumption \ref{ass2} to be satisfied for an arbitrary closed set  in a (directed) stoichiometric compatibility class:
\begin{itemize}
\item[(a)] $\mRf\cup (\mRb\setminus\mRa)$ is weakly reversible.
\item[(b)] $\mRf\cap \mRa\cap\mRb$ is weakly reversible, and for any reaction $r\in \mRb\setminus\mRa$, the inverse reaction is in $\mRf$.
\item[(c)] $\mN$ is weakly reversible and $\mF=\mRa$.
\end{itemize}
\end{lemma}

\begin{proof}
As $\mN$ is sub-conservative, any closed set of a directed stoichiometric compatibility class is finite.
Hence, it is enough to show that  \eqref{eq:fin_out3} holds for arbitrary $x=(z,0)\in\Z_{\geq0}^p\times \{0\}$. 
As  \eqref{eq:lambda} holds by assumption, it is enough to show that for  arbitrary $r\in \mRb\setminus \mRa$ with $ \lam_{r}(z,0)\neq 0$ the following holds:
$$ \sum\limits_{\w\in \mWu\st \w\text{ starts with }r}\Lambda_{\w}(z)= \lam_{r}(z,0).$$

Consider the SRN with species set $\mS$ and reactions $\mF$, denoted $\mN|_\mF$, that we take with the same transition intensities as $\mN$.  Consider the associated discrete-time Markov chain (DTMC) from the jump chain of $\mN|_\mF$ (cf. \cite[p. 82]{norris}), which is the sequence of states taken by the CTMC. This will be denoted by $(Y_n)_{n\in\N}$.
By definition of $\Lambda_{\w}(\cdot)$, $\frac{\Lambda_{\w}(z)}{\lam_{r}(z,0)}$ corresponds to the probability that the transitions of the jump chain $(Y_n)_{n\in\N}$ when starting from $(z,0)+y'-y$ comes from the sequence of reactions in $\w$.
 
 Denote the set of reachable states from $(z,0)+y'-y$ via $\mN|_\mF$ by $\mN|_\mF((z,0)+y'-y)$. As $\mN$ is sub-conservative, $\mN|_\mF$ is sub-conservative and $\mN|_\mF((z,0)+y'-y)$ is finite. 
 In the following, let $e_i\in \Z^n_{\geq 0}$ denote the vector with $U_i$-coordinate entry equal to one and all other entries equal to zero.

(a) Assume $\mRf\cup (\mRb\setminus\mRa)$ is weakly reversible. Then, if $(z\p,0)+e_i\in \mN|_\mF((z,0)+y'-y)$, then also $(z,0)\in \mN|_\mF((z\p,0)+e_i)$. Hence, any such $(z\p,0)+e_i$ is transient in the  jump chain $(Y_n)_{n\in\N}$ of $\mN|_\mF$ . On the other hand by definition, any state of the form $(z',0)$ is absorbing for the jump chain of $\mN|_\mF$. As the set of reachable states is finite, the time until absorption is a.s. finite. So overall the jump chain ends in one of the absorbing states after a finite time, and 
$$\sum\limits_{\w\in \mWu\st \w\text{ starts with }r}\frac{\Lambda_{\w}(z)}{\lam_{r}(z,0)}=1.$$

For (b) and (c), the proofs are similar to (a).
\end{proof}

\color{black}

\section{Proof of Theorem \ref{main_thm_stand}}\label{app}

The proof of Theorem \ref{main_thm_stand} is based on singularly perturbed and watched Markov chains. We therefore first introduce these concepts for the convenience of the reader.

\subsection{Singularly perturbed CTMCs}\label{sing_pert}
Singular perturbation theory for Markov chains is a well-developed topic in probability theory.
For the convenience of the reader we restate and summarize results of \cite{yin_zhang_trans,yin2012continuous} for homogeneous CTMCs on a finite state space, using their notation. For more on the two-scale setting for CTMCs, we refer to \cite{pavli,yin_zhang_trans,yin2012continuous}. 

 The setting we introduce has a Markov chain with fast and slow components subject to so-called weak and strong interactions, where the fast components consist of absorbing, weakly irreducible or transient states of the fast dynamics. These cases are treated in detail in \cite[$\S$ 4.3, $\S$ 4.4, $\S$ 4.5]{yin2012continuous}. 

Let a sequence of homogeneous generators of a CTMC be given by
$$Q^\eps=\frac{1}{\eps}\widetilde{Q}+\widehat{Q},\quad \eps>0,$$
on a finite state space $\I$. By rearrangement, the states can be divided into blocks of states $\I_A,\I_W,\I_T$ according to their role in $\widetilde{Q}$, which are absorbing, irreducible, and transient, respectively. The matrices $\widetilde{Q},\widehat{Q}$ are partitioned accordingly into block-matrices and take the form
\begin{equation}\label{form_matrices_Yin}
\widetilde{Q} =
\begin{pmatrix}
\widetilde{Q}_{AA} & \widetilde{Q}_{AW} & \widetilde{Q}_{AT} \\
\widetilde{Q}_{WA} & \widetilde{Q}_{WW} & \widetilde{Q}_{WT} \\
\widetilde{Q}_{TA} & \widetilde{Q}_{TW}  &\widetilde{Q}_{TT} 
\end{pmatrix}=\begin{pmatrix}
0 & 0 & 0 \\
0 & \widetilde{Q}_{WW} & 0 \\
\widetilde{Q}_{TA} & \widetilde{Q}_{TW}  &\widetilde{Q}_{TT} 
\end{pmatrix},
\end{equation}
\begin{equation*}
\widehat{Q} =
\begin{pmatrix}
\widehat{Q}_{AA} & \widehat{Q}_{AW} & \widehat{Q}_{AT} \\
\widehat{Q}_{WA} & \widehat{Q}_{WW} & \widehat{Q}_{WT} \\
\widehat{Q}_{TA} & \widehat{Q}_{TW}  &\widehat{Q}_{TT} 
\end{pmatrix}.
\end{equation*}

The states in $\I_A$ with $\widetilde{Q}_{AA}= 0,\widetilde{Q}_{AW}=0,\widetilde{Q}_{AT}=0$ correspond to absorbing states of $\widetilde{Q}$. The states in $\I_T$ correspond to transient states of $\widetilde{Q}$ (hence $\widetilde{Q}_{TT}$ is a Hurwitz matrix, furthermore stable and non-singular see Lemma \ref{trans_stable_non}). Finally, $\I_W$ consists of the states that are part of (non-trivial) closed irreducible components, so $ \widetilde{Q}_{WW}$ in \eqref{form_matrices_Yin} can  further be decomposed as
\begin{equation}\label{form_irred_mat}
\widetilde{Q}_{WW} = \text{diag}(\widetilde{Q}^1,\widetilde{Q}^2,\cdots, \widetilde{Q}^l)=
\begin{pmatrix}
\widetilde{Q}^1 & & & \\
 &\widetilde{Q}^2& & \\
  && \ddots & \\
   && & \widetilde{Q}^l
\end{pmatrix},
\end{equation}
where $\widetilde{Q}^i$ is an $m_i\times m_i$ matrix.
We denote by $\nu^i$ the stationary distribution  of $\widetilde{Q}^i$ for $i=1,\ldots,l$ (which exist by assumptions on the state space).
\begin{remark}
 Our presentation differs slightly from 
\cite{yin2012continuous}. In that work, so-called weakly irreducible classes is used \cite[Definition 2.7a]{yin2012continuous}; classes containing exactly one closed irreducible component and possibly other states that lead to this component (which are then transient) in the case of time-homogeneous CTMCs. 
As we include transient states of $\widetilde{Q}$ via $\I_T$, this is not a  restriction. We note that our setting is included in \cite[$\S$ 4.5]{yin2012continuous}.
\end{remark}

For $\epsilon>0$, the forward equation of the CTMC gives an ODE
\begin{equation*}
 \frac{p^\epsilon(t)}{dt}=p(t)Q^\epsilon, \quad p(0)=\pi_0,
\end{equation*}
where $\pi_0$ is the initial probability distribution.

In \cite{yin2012continuous}, they construct sequences of functions that approximate $p^\eps(t)$ in the limit for $\eps\to 0$ uniformly on $[0,T]$ for different powers of $\eps$. One contribution comes from the so-called outer expansion, that approximates $p^\eps(t)$ for $t>0$. Another part comes from the so-called initial-layer correction, that approximates $p^\eps(t)$ in a neighborhood of $t=0$. This might be ignored as it is zero in our case \cite{yin2012continuous}. Hence, we will focus on the zeroth order outer expansion as this is enough to obtain an approximation to order $O(\epsilon)$ (see Proposition \ref{prop_main1}).

Following \cite[
Theorem 4.45]{yin2012continuous}, the distribution of the CTMC  converges to the zeroth order outer expansion $\varphi_0(t)$ for $t>0$. In the case we consider, it is given as 
\begin{equation*}
\begin{split}
\varphi_0(t) &= (\varphi_A(t),\varphi_W(t),\varphi_T(t))=(\varphi_A(t),\varphi_W(t),0_T), \\
&=(\varphi_A(t),\vartheta^1(t)\nu^1,\cdots, \vartheta^l(t)\nu^l,0_T),
\end{split}
\end{equation*}
where  $\vartheta^i(t)$ are scalar functions and   $\nu^i$ are the stationary distributions from matrix \eqref{form_irred_mat}, and $0_T$ is the vector of zero-entries.
Then, let  (with $m_0=0$)  
\begin{equation*}
\widetilde{\1}=\text{diag}(\widetilde{\1}_{m_1},\cdots,\widetilde{\1}_{m_l} )
\end{equation*}
where 
\begin{equation*}
\widetilde{\1}_{m_i} =(\underbrace{0,\cdots,0}_{m_1+\cdots+ m_{i-1}}\underbrace{1,\cdots,1}_{m_i},\underbrace{0,\cdots,0}_{m_{i+1}+\cdots +m_l})^\top, \quad i=1,\ldots,l,
\end{equation*}
and define the matrix \begin{equation}\label{eq_gen_red}
\overline{Q}_\ast=\text{diag}(\underbrace{1,\cdots,1}_{|\I_A|},\nu^1,\cdots ,\nu^l)
\bigg (
\begin{pmatrix}
\widehat{Q}_{AA} & \widehat{Q}_{AW}  \\
\widehat{Q}_{WA} & \widehat{Q}_{WW} 
\end{pmatrix}+
\begin{pmatrix}
\widehat{Q}_{AT}   \\
\widehat{Q}_{WT}  
\end{pmatrix}
\widetilde{Q}^{-1}_{TT} 
\begin{pmatrix}
\widetilde{Q}_{TA}   & \widetilde{Q}_{TW}  
\end{pmatrix}
\bigg) \text{diag}(\underbrace{1,\cdots,1}_{|\I_A|},\widetilde{\1}).
\end{equation}
Note that $\overline{Q}_\ast$ is not the same as $Q^0$ (which was previously defined for SRNs).

\begin{remark}
The matrices in the product above have different dimensions, i.e. $\text{diag}(1,\cdots,1,\nu^1,\cdots ,\nu^l)$ is a $(|\I_A|+l)\times (|\I_A|+|\I_W|)$ matrix, the middle matrix is a $ (|\I_A|+|\I_W|)\times (|\I_A|+|\I_W|)$ matrix whereas $\text{diag}(1,\cdots,1,\widetilde{\1})$ is a $(|\I_A|+|\I_W|)\times (|\I_A|+l)$ matrix.
\end{remark}

Then the zeroth order outer expansion $\varphi_0(t)$ (corresponding to \cite[(4.86)]{yin2012continuous}) with initial distribution $\pi_0=(p_A,p_W,p_T)$ is determined by the following ODE,
\begin{equation}\label{eq_ODE_general}\left\{
\begin{split}
\frac{d}{dt}(\varphi_A(t),\vartheta^1(t),\cdots, \vartheta^l(t)) &=(\varphi_A(t),\vartheta^1(t),\cdots, \vartheta^l(t))\overline{Q}_\ast, \\
\varphi_A(0) &= p_A + p_T(\widetilde{Q}_{TT})^{-1}\widetilde{Q}_{TA},\\
(\vartheta^1(0),\cdots, \vartheta^l(0)) &= \big (p_W+p_T(\widetilde{Q}_{TT})^{-1}\widetilde{Q}_{TW} \big)\widetilde{\1},
\end{split}\right.
\end{equation}
which defines the generator of a CTMC that is given by $\overline{Q}_*$. This CTMC has reduced state space $\I_A\cup\{s_{1},\cdots ,s_{l}\}$,  where each state $s_i$ corresponds to the assembled irreducible component,  which has stationary distribution $\nu^i$  from $\widetilde{Q}_{WW}$,  see  \eqref{form_irred_mat}. Then, denoting the projection onto the reduced state space of the initial value of   \eqref{eq_ODE_general}
by $\text{Pr}(\pi_0):=(\varphi_A(0),\vartheta^1(0),\cdots, \vartheta^l(0))$, we get a solution 
$$(\varphi_A(t),\vartheta^1(t),\cdots, \vartheta^l(t))=\text{Pr}(\pi_0)e^{\overline{Q}_\ast t}.$$
Denote the sequence of CTMCs associated to $Q^\eps$ by $X_\eps(t)$ and the CTMC from $\overline{Q}_\ast$ by $X_0(t)$. 
Under these conditions the following holds \cite[Theorem 4.45]{yin2012continuous}.

\begin{proposition}\label{prop_main1}
Let $\pi_0$ be a probability distribution on $\I_A\cup \I_W\cup\I_T$ with support on $\I_A$.  Then, for all $T>0$ and all $B\subseteq \I_A$:
$$\sup_{t\in[0,T]}|P_{\pi_0}(X_\eps(t)\in B)-P_{\text{Pr}(\pi_0)}(X_0(t)\in B)|=O(\epsilon),\quad\text{for }\epsilon\to 0.$$
\end{proposition}

The contribution of the zeroth order outer expansion is sufficient to approximate the above events to order $O(\epsilon)$ as the initial-layer corrections are zero \cite{yin2012continuous}. Hence,  the CTMC generated by $\overline{Q}_*$ can be seen as the limit CTMC (in the sense of Proposition \ref{prop_main1}). The above will be the main reference of this section for the proof of Theorem \ref{main_thm_stand}.

\begin{remark}
 \label{rem_time_dep}
In the inhomogeneous case, the generator of the scaled CTMC of \cite{yin2012continuous} has the form $Q^\eps_t=\frac{1}{\eps}\widetilde{Q}_t+\widehat{Q}_t$, where $\widetilde{Q}_t,\widehat{Q}_t$ satisfy some regularity conditions on $t\in [0,T]$. The conclusions we have stated are essentially the same in this case. In particular,  the error $O(\epsilon)$ for the zeroth order outer expansion is as in Proposition \ref{prop_main1}.
\end{remark}

\subsection{Watched CTMCs}\label{stoch_compl}

We introduce watched CTMCs on a finite state space. Watched Markov chains appear as a restriction of a bigger Markov chain, which we observe only if the chain is in a specific subset of the state space \cite{kemenyfinite,Meyer}. 

In the following, we consider a CTMC  $(X(t))_{t\ge 0}$  with generator $Q$ on a finite state space $S$ with $E_1\subseteq S$ a subset of the state space.
Writing the $Q$-matrix in block-form according to the sets $E_1$ and $E_2 =S\setminus E_1$, we decompose $Q$ as  
\begin{equation*}\label{form_start}
Q =
\begin{pmatrix}
Q_{E_1E_1} & Q_{E_1E_2} \\ Q_{E_2E_1} & Q_{E_2E_2}
\end{pmatrix}.
\end{equation*}

\begin{lemma} \label{trans_stable_non}
If under $$Q\p:=
\begin{pmatrix}
Q_{E_1E_1} & Q_{E_1E_2} \\ 0 & 0
\end{pmatrix}
$$
 all states in $E_1$ are transient, 
then $Q_{E_1E_1}$ is stable and non-singular.
\end{lemma}
\begin{proof}
Note that no diagonal entries of $Q_{E_1E_1}$ are zero. Since all states are transient, then all communicating component of the matrix have a ``leak" (the row sum of $Q_{E_1E_1}$for a least one state in a component  is non-zero) and the result follows from \cite{PLEMMONS1977175}.
\end{proof}

The censored Markov chain on $E_2$ is defined as the process that has sample paths following the process $(X(t))_{t\ge 0}$  as long as it is in $E_2$, and ignoring the parts where $X(t)\in E_1$. This so-defined process is then again a Markov process, 
where under the assumption above, the transition rates of the watched Markov chain are given by
$$Q_{\text{watched}}:= Q_{E_2E_2}-Q_{E_2E_1}Q_{E_1E_1}^{-1}Q_{E_1E_2}.$$

\begin{remark}\label{connection_reduction}
In the above formula, the second summand catches the contribution from when  the original CTMC enters $E_1$  until it returns to $E_2$. By definition of the watched CTMC, we have
$$(Q_{E_2E_1}Q_{E_1E_1}^{-1}Q_{E_1E_2})_{v,w}=\sum_{x\in E_1}q_{v,x}P_{x}(X_{\tau}=w),$$
where $\tau=\inf\{t\geq0|X(t)\in E_2\}$.
\end{remark}

\subsection{Proof of Theorem \ref{main_thm_stand}}
\label{sect_proof_stand}

The proof consists of several parts. We first  show  without loss of generality, the following might be assumed:

\begin{itemize}
\item[-] the closed set $E$ has a particular form, 
\item[-] the initial distribution has support on a special part of the state space, 
\item[-] the RN $\mN$ is finite. 
\end{itemize}
Then we apply singularly perturbed CTMCs of $\S$ \ref{sing_pert} to the SRN. As a third step, we simplify by showing, the dynamics can be restricted to a smaller part of the state space. Finally, we connect the limit CTMC to watched CTMCs of $\S$ \ref{stoch_compl}, which shows how the reduced reactions are derived from the limit CTMC of the scaled SRNs.

\subsubsection{Preparations}
\label{compl_comparison}

In preparation for the proof, we consider the following. Let $\mN$ be an SRN with  non-interacting species set $\mU$, fast proper reactions $\mF\subseteq\mR_\mU$, and transition intensities $\lambda_r, r\in \mR$.
 For $x\in\Z^n_{\geq0}$, denote by $\mN(x)$ the set of reachable states from $x$ of $\mN$, and for  $z\in\Z^p_{\geq0 }$, denote by $\mNuf(z)$  the set of reachable states from $z$ of $\mNuf$.

We introduce a \emph{surrogate} model $\mathcal{M}$ that reflects the behavior of the SRN of $\mN$ when $\eps$ is very small. Specifically, this model is a CTMC on $\Z^n_{\geq 0}$ determined by the $Q$-matrix
$$Q(x,x+\xi):=\sum_{r\colon y \to y\p\in \mR\colon y\p-y=\xi}\tilde{\lam}_r(x),$$
with transition intensities,
$$\tilde{\lam}_r(x):=
\begin{cases}
\lam_r(x) & \text{ if $r\in \mF$},\\
  \1_{N_0}(x)\lam_r(x)& \text{ if $r\in \mR\setminus \mF$}.
\end{cases} $$
 Similarly, for $x\in\Z^n_{\geq0}$, we denote by $\mathcal{M}(x)$  the set of reachable states from $x$ of  $\mathcal{M}$. 

Let $N_i:=\{x=(z,u)\in \Z_{\geq 0}^n|\sum_{j=1}^m u_j=i\}$ for $i\geq 0$.
Then, we have the following.

\begin{proposition}\label{impl}
Suppose Assumption \ref{ass1} holds, that $\mF$ is proper, and let $x=(z,0)\in N_0$. Then,
$\rho_{\mO}(\mathcal{M}(x)\cap N_0))=\mNuf(z)$ and
$\mathcal{M}(x)\subseteq\mN(x)$
\end{proposition}

For the proof of Theorem \ref{main_thm_stand}, 
the following Lemma will be useful. The proof follows by contradiction and is omitted.
\begin{lemma}\label{transient_+}
Suppose  Assumption \ref{ass1} holds and that $\mF$ is proper. If Assumption \ref{ass2} holds for some closed set $E$, then for any $x\in E\cap N_0$, the states in $\mathcal{M}(x)\cap N_1$ are transient for the CTMC of the SRN obtained by considering the RN defined by the reaction set  $\mF$ (and taking the same transition intensities as in $\mN$).
\end{lemma}

Consider the family of scaled CTMCs as in the setting  of $\S$ \ref{reduced_dynamics} with notation introduced in $\S$ \ref{compl_comparison}. Furthermore, consider a finite closed set $E\subseteq \Z_{\geq 0}^{n}$. As $E$ is finite, we can assume wlog that $\mN$ is finite.

For any probability distribution $\pi$ on $E$ and any state $w\in E$, $P_\pi(X(t)=w)=\sum_{x\in E}\pi_xP_x(X(t)=w)$. 
Therefore it is enough to prove the statement for:
\begin{itemize}
\item[-]  a closed set of the form $\tilde{E}:=\mN(x)$ with $x\in N_0\cap E$. This follows as $E=\cup_{x\p\in E\cap N_0}\mN(x)$ by definition.
\item[-] an initial distribution with support on $\mathcal{M}(x\p)\cap N_0$ for $x\p\in \tilde{E}\cap N_0$. This follows as by definition $\tilde{E}\cap N_0=(\cup_{x\p\in \tilde{E}\cap N_0}\mathcal{M}(x\p))\cap N_0$, see Proposition \ref{impl}. 
\end{itemize}

Hence, we might assume that $\mN$ is finite, that $E:=\mN(x)$ for some $x\in N_0$ and that $\pi$ has support on $\mathcal{M}(x\p)\cap N_0\subseteq E$ for $x\p\in E$. We denote $\Eu:=\mathcal{M}(x\p)$.

\subsubsection{The proof}

With the preparations above, the matrix $\widetilde{Q}$ from $\S$ \ref{sing_pert} is comprised of the transition intensities from reactions of $\mRf$ with a non-interacting species of $\mU$ in the reactant and $\widehat{Q}$ is comprised of the other reactions $\mR\setminus \mRf$, where we divide the states of $E$ as follows:
$$ 
\I_{S_1}=\Eu\cap N_0,\quad
\I_{S_2}=(E \cap N_0)\setminus \Eu,\quad
$$
$$
\I_{F_1}=\Eu\cap N_1,\quad
\I_{F_2}=E\setminus (N_0\cup (\Eu\cap N_1)),
$$
where by definition $E=\I_{S_1}\cup\I_{S_2}\cup\I_{F_1}\cup\I_{F_2}$.
We next illustrate the possible slow transitions (that is, transitions from $\widehat{Q}$) outgoing from $\I_{S_1}$ in blue on the left and all possible fast transitions (that is, transitions from $\widetilde{Q}$) displayed in red on the right by their transition state diagrams.
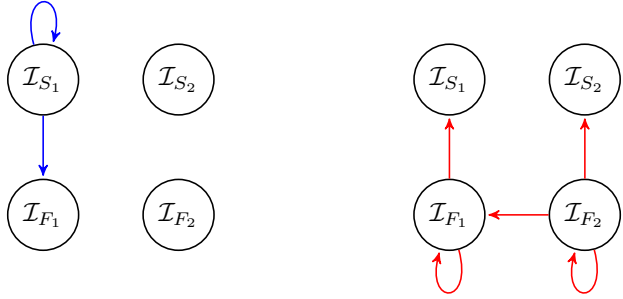
\begin{center}
\begin{tikzpicture}[->,>=stealth',shorten >=1pt,auto,node distance=1.8cm,
                    semithick, scale=0.6]
  \tikzstyle{every state}=[fill=white,draw=black,text=black]
    \tikzstyle{every edge}=[draw=blue]

  \node[state] (A)                    {$\I_{S_1}$};
  \node[state]         (B) [right of=A] {$\I_{S_2}$};
  \node[state]         (C) [below of=A] {$\I_{F_1}$};
  \node[state]         (D) [right of=C] {$\I_{F_2}$};

  \path (A) edge              node {} (C)
           (A)  edge       [loop above]       node {} (C);

          \tikzstyle{every edge}=[draw=red]
          
                            \node[state,draw=none] (A0)[right of=B]{};
                                      
                  \node[state] (A1)[right of=A0]{$\I_{S_1}$};
  \node[state]         (B1) [right of=A1] {$\I_{S_2}$};
  \node[state]         (C1) [below of=A1] {$\I_{F_1}$};
  \node[state]         (D1) [right of=C1] {$\I_{F_2}$};

  \path (C1) edge              node {} (A1)
           (C1)  edge       [loop below]       node {} (C1)
        (D1) edge node {} (B1)
            edge   [loop below]            node {} (D)
        (D1) edge              node {} (C1);

\end{tikzpicture}
\end{center}

These transition diagrams determine the zero/non-zero parts of the corresponding $Q$-matrices in block-form as follows.

\begin{lemma}\label{lemma_sl_fst} The following holds for the fast and slow parts of the  $Q$-matrices (notation as in  $\S$ \ref{sing_pert}) divided into blocks:
\begin{equation}\label{eq_big_Q_mat}
\widehat{Q}|_{\I_{S_1}\times S} =
\begin{pmatrix}
\widehat{Q}_{S_1S_1} &\widehat{Q}_{S_1S_2} &\widehat{Q}_{S_1F_1} &\widehat{Q}_{S_1F_2}  
\end{pmatrix}
=
\begin{pmatrix}
\widehat{Q}_{S_1S_1} &0 &\widehat{Q}_{S_1F_1} &0  
\end{pmatrix},
\end{equation}

\begin{equation*}
\widetilde{Q} =
\begin{pmatrix}
\widetilde{Q}_{S_1S_1} &\widetilde{Q}_{S_1S_2} &\widetilde{Q}_{S_1F_1} &\widetilde{Q}_{S_1F_2}  \\ 
\widetilde{Q}_{S_2S_1} &\widetilde{Q}_{S_2S_2} &\widetilde{Q}_{S_2F_1} &\widetilde{Q}_{S_2F_2}  \\
\widetilde{Q}_{F_1S_1} &\widetilde{Q}_{F_1S_2} &\widetilde{Q}_{F_1F_1} &\widetilde{Q}_{F_1F_2}  \\ 
\widetilde{Q}_{F_2S_1} &\widetilde{Q}_{F_2S_2} &\widetilde{Q}_{F_2F_1} &\widetilde{Q}_{F_2F_2}  
\end{pmatrix}=\begin{pmatrix}
0 &0 &0 &0  \\ 
0 &0 &0 &0  \\ 
\widetilde{Q}_{F_1S_1} &0 &\widetilde{Q}_{F_1F_1} &0 \\ 
\widetilde{Q}_{F_2S_1} &\widetilde{Q}_{F_2S_2} &\widetilde{Q}_{F_2F_1} &\widetilde{Q}_{F_2F_2}  
\end{pmatrix}
\end{equation*}
\end{lemma}

\begin{proof}
Recall that the slow transitions come from $\mR\setminus \mRf$ and the fast ones from $ \mRf$ by assumption. We go through the slow transitions and then the fast ones.
\begin{itemize}
\item[-] By contradiction there are no slow transitions from $\I_{S_1}$ to $\I_{S_2}$ or $\I_{F_2}$. 
\item[-]  By Assumption \ref{ass1}, $\widetilde{Q}$ has zero entries on $E_0=\I_{S_1}\cup\I_{S_2}$ (the coordinates of non-interacting species are zero). 
By the definition of the sets $\I_{F_1},\I_{S_2}$ and $\I_{F_2}$ there are no fast transitions from $\I_{F_1}$ to $\I_{S_2}$ or $\I_{F_2}$.
\end{itemize}
\end{proof}
Following $\S$ \ref{sing_pert} (see also \cite{yin2012continuous}), the distribution of the CTMC converges to the zeroth order outer expansion $\varphi_0(t)$ for $t\geq 0$, which we consider as 
$$\varphi_0(t):=(\varphi_{S_1}(t),\varphi_{S_2}(t),\varphi_{F_1}(t),\varphi_{F_2}(t)).$$
By assumption, the initial distribution $\pi_0$ has support contained in $\I_{S_1}$. 
We next go through the roles of the states in $\I_{S_1},\I_{S_2},\I_{F_1},\I_{F_2}$.

States in $\I_{F_1}$ are transient in $\widetilde{Q}$ by Assumption \ref{ass2} and Lemma \ref{transient_+}, and $\widetilde{Q}|_{\I_{F_1}\times \I_{F_1}}$ is non-singular (see Lemma \ref{trans_stable_non}). States in $\I_{S_1},\I_{S_2}$ are absorbing in $\widetilde{Q}$. States in $\I_{F_2}$ can be transient, absorbing or part of a closed communicating class in $\widetilde{Q}$. Correspondingly, $\varphi_0(t)$ is zero on $\I_{F_1}$ and on the transient part of $\I_{F_2}$.

Next, we make the following general observation.
\begin{remark}\label{poss_red}
Let $(X(t))_{t\ge 0}$ be a CTMC with generator $Q$ on a state space $S$. Let $\pi_0$ be a probability distribution with $\supp(\pi)\subseteq Z\subseteq S$, where $|Z|<\infty$. Assume that there are no $x\in Z,y\in S\setminus Z$ such that $x\to y$. Then, $Q|_{Z\times Z}$ is a $Q$-matrix,  and denoting by $X^{restr}$ the corresponding CTMC, we have 
$$P_{\pi_0}(X(t)\in B)=P_{\pi_0}(X^{restr}(t)\in B),\quad \text{ for }B\subseteq S,\quad t\geq 0.$$
Then, it is enough to consider the restricted $Q$-matrix $Q|_{A\times A}$ and $X^{restr}$ for computations of such probabilities.
\end{remark}

\begin{lemma}
The CTMC of the generator $\overline{Q}_*$ (see \eqref{eq_gen_red} of $\S$ \ref{sing_pert}) in our setting has the property of Remark \ref{poss_red} with $Z:=\I_{S_1}=\Eu\cap N_0$.
\end{lemma}
\begin{proof}
First, we decompose $\I_{F_2}$ into three different parts $F_{2,t},F_{2,w}$ and $F_{2,a}$ according to whether a state is transient, part of a (non-trivial) irreducible component or absorbing in $\widetilde{Q}$. Then, writing 
\begin{equation}\label{eq_deco}
A=\I_{S_1}\cup\I_{S_2}\cup F_{2,a}, \quad T=\I_{F_1}\cup F_{2,t},\quad W=F_{2,w},
\end{equation} corresponds to the situation of $\S$ \ref{sing_pert}.

Let $S\p$ be the reduced state space of $\overline{Q}_*$ (cf.,  $\S$ \ref{sing_pert}). Clearly $A\subseteq S\p$, as these are the absorbing states in $\widetilde{Q}$, and at least the transient states have been eliminated from $S\p$. Hence $\I_{S_1}\subseteq S\p$. We next focus on $\overline{Q}_*|_{\I_{S_1}\times S\p}$, which contains  the possible outgoing transitions from $\I_{S_1}$. It is enough to show that $\overline{Q}_*|_{\I_{S_1}\times S\p\setminus\I_{S_1}}$ has only zero entries. For this, it suffices to consider the matrices in the middle of  \eqref{eq_gen_red}, which we recall here, 

\begin{equation}\label{matrix_res}
\bigg (
\begin{pmatrix}
\widehat{Q}_{AA} & \widehat{Q}_{AW}  \\
\widehat{Q}_{WA} & \widehat{Q}_{WW} 
\end{pmatrix}+
\begin{pmatrix}
\widehat{Q}_{AT}   \\
\widehat{Q}_{WT}  
\end{pmatrix}
\widetilde{Q}^{-1}_{TT} 
\begin{pmatrix}
\widetilde{Q}_{TA}   & \widetilde{Q}_{TW}  
\end{pmatrix}
\bigg).
\end{equation}
We next check the right summand of  \eqref{matrix_res}, which we denote
$$R=\begin{pmatrix}
\widehat{Q}_{AT}   \\
\widehat{Q}_{WT}  
\end{pmatrix}
\widetilde{Q}^{-1}_{TT} 
\begin{pmatrix}
\widetilde{Q}_{TA}   & \widetilde{Q}_{TW}  
\end{pmatrix}=\begin{pmatrix}
\widehat{Q}_{AT}\widetilde{Q}^{-1}_{TT} \widetilde{Q}_{TA} &\widehat{Q}_{AT}\widetilde{Q}^{-1}_{TT} \widetilde{Q}_{TW}  \\
\widehat{Q}_{WT}\widetilde{Q}^{-1}_{TT} \widetilde{Q}_{TA} &\widehat{Q}_{WT}\widetilde{Q}^{-1}_{TT} \widetilde{Q}_{TW}
\end{pmatrix}.$$
Then, consider the restriction to $R|_{\I_{S_1}\times S\p}$. Keeping the decomposition of $A,T,W$ in \eqref{eq_deco} in mind, the following hold for the involved matrices $\widehat{Q}_{AT},$ $\widetilde{Q}^{-1}_{TT},$ $\widetilde{Q}_{TA},$  $\widetilde{Q}_{TW}$ by Lemma \ref{lemma_sl_fst}:
$$\widetilde{Q}_{TT}^{-1} =
\begin{pmatrix}
\widetilde{Q}_{F_1F_1} & 0 \\ \ast & \ast
\end{pmatrix}^{-1}=\begin{pmatrix}
\widetilde{Q}_{F_1F_1}^{-1} & 0 \\ \ast &\ast
\end{pmatrix},\quad
\widetilde{Q}_{TW}=\begin{pmatrix}
\widetilde{Q}_{F_1F_{2,w}} \\ \widetilde{Q}_{F_{2,t}F_{2,w}}\end{pmatrix}=
\begin{pmatrix}
0\\ \ast 
\end{pmatrix},
$$

$$ 
 \widehat{Q}_{AT}=\begin{pmatrix}
\widehat{Q}_{S_1F_1} & 0\\ \ast &\ast\\
\ast &\ast
\end{pmatrix},\quad 
 \widetilde{Q}_{TA}=\begin{pmatrix}
\widetilde{Q}_{F_1S_1} & 0&0\\ \ast &\ast
&\ast
\end{pmatrix},
$$
and we get that
\begin{align*}R|_{S_1\times S\setminus T}=& \begin{pmatrix} R_{S_1S_1}&R_{S_1S_2}&R_{S_1F_{2,a}}&R_{S_1S_{2,w}}\end{pmatrix}\\
&\begin{pmatrix} \widehat{Q}_{S_1F_1}\widetilde{Q}_{F_1F_1}^{-1}\widetilde{Q}_{F_1S_1}&0&0&0\end{pmatrix} \stepcounter{equation}\tag{\theequation}\label{matsum1}\\
\end{align*}
Hence, the property holds for $R$, that is, the right hand side of \eqref{matrix_res} has the required property.

Next we look at the left hand side of  \eqref{matrix_res}, and again restrict to $\I_{S_1}\times S\setminus T$. By Lemma \ref{lemma_sl_fst}, the slow part has only non-zero transitions to $\I_{S_1}$ or $\I_{F_1}$, hence also this matrix has the required property. As this holds for both summands of    \eqref{matrix_res}, we have shown that the property holds for $\overline{Q}_*$.
\end{proof}

Finally, noting the form of equation \eqref{eq_big_Q_mat} and equation \eqref{matsum1} we see that only the slow and fast parts outgoing from $\I_{F_1}$ and $\I_{S_1}$ contribute to $\overline{Q}_*|_{\I_{S_1}\times \I_{S_1}}$. Then we can compute $\overline{Q}_*$ for the state space $\I_{F_1}\cup \I_{S_1}$ (i.e. with restricted $Q_\eps,\widetilde{Q},\widehat{Q}$) or for the state space  $\I_{S_1}\cup\I_{S_2}\cup\I_{F_1}\cup\I_{F_2}$. However, in both cases the expression we get for $\overline{Q}_*|_{\I_{S_1}\times \I_{S_1}}$ is the same, hence the following holds.

\begin{lemma}\label{lem_simpli}
If the initial distribution has support on $\I_{S_1}$, for the computation of $\varphi_0(t)$, it is enough to restrict the state space of $Q_\eps$ to $\I_{F_1}\cup \I_{S_1}$.
\end{lemma}

We treat the restriction of the state space to $\I_{F_1}\cup \I_{S_1}$ in the following. Restricting to $\I_{F_1}\cup \I_{S_1}$, we have
\begin{equation*}
\widetilde{Q} =
\begin{pmatrix}
\widetilde{Q}_{S_1S_1} & \widetilde{Q}_{S_1F_1} \\ \widetilde{Q}_{F_1S_1} & \widetilde{Q}_{F_1F_1}
\end{pmatrix}=\begin{pmatrix}
0 & 0 \\ \widetilde{Q}_{F_1S_1} & \widetilde{Q}_{F_1F_1}
\end{pmatrix},\;\;\;
\widehat{Q} =
\begin{pmatrix}
\widehat{Q}_{S_1S_1} & \widehat{Q}_{S_1F_1} \\ \widehat{Q}_{F_1S_1} & \widehat{Q}_{F_1F_1}
\end{pmatrix}.
\end{equation*}
where the states $\I_{S_1}$ are absorbing in $\widetilde{Q}$, and the states of $\I_{F_1}$ are transient states. Note that this setting corresponds to the situation of $\S$ \ref{sing_pert} where we removed the irreducible part (that is, $\I_W$ in the notation of $\S$ \ref{sing_pert}), $\I_{F_1}$  corresponds to $T$ and $\I_{S_1}$  corresponds to $A$ of $\S$ \ref{sing_pert}.

The scaled CTMC converges to the zeroth order outer expansion $\varphi_0(t)$ for $t> 0$, giving
$\varphi_0(t):=(\varphi_{S_1}(t),\varphi_{F_1}(t))=(\varphi_{S_1}(t),0_{F_1})$ which satisfies the ODE from $\S$ \ref{sing_pert} with 
\begin{equation}\label{eq_ODE}\left\{
\begin{split}
\dot{\varphi}_{S_1}(t) &= \varphi_{S_1}(t)(\widehat{Q}_{S_1S_1}+\widehat{Q}_{S_1F_1}(\widetilde{Q}_{F_1F_1})^{-1}\widetilde{Q}_{F_1S_1}), \\
\varphi_{S_1}(0) &= p_{S_1} + p_{F_1}(\widetilde{Q}_{F_1F_1})^{-1}\widetilde{Q}_{F_1S_1} = Pr(\pi_0),
\end{split}\right.
\end{equation}
with initial distribution $\pi_0=(p_{S_1},p_{F_1})$. 
As the initial distribution has support on $\I_{S_1}$, we get the following from Proposition \ref{prop_main1}, where again $X_\eps(t)$ is the CTMC associated to $Q_\eps$ while $X_0(t)$ is the CTMC from $C$ (corresponding to $\overline{Q}_\ast$ from $\S$ \ref{sing_pert}).

\begin{lemma}
Let $\pi_0$ be a probability distribution with support on $\I_{S_1}$, and $B\subseteq \I_{S_1}$. Then the following holds for all $T>0$:
$$\sup_{t\in[0,T]}|P_{\pi_0}(X_\eps(t)\in B)-P_{Pr(\pi_0)}(X_0(t)\in B)|=O(\epsilon)\quad\text{for }\epsilon\to 0.$$
\end{lemma}

Finally, observing the form of the matrix $\widehat{Q}_{S_1S_1}+\widehat{Q}_{S_1F_1}(\widetilde{Q}_{F_1F_1})^{-1}\widetilde{Q}_{F_1S_1}$ in  \eqref{eq_ODE}, we can view it as the watched Markov chain of $\S$ \ref{stoch_compl}, where it is watched when in $\I_{S_1}$ with $Q$-matrix
\begin{equation*}
Q =
\begin{pmatrix}
\widehat{Q}_{S_1S_1} & \widehat{Q}_{S_1F_1} \\ \widetilde{Q}_{F_1S_1} & \widetilde{Q}_{F_1F_1}
\end{pmatrix}.
\end{equation*}
By the interpretation of the censored Markov chain available from Remark \ref{connection_reduction}, the part $\widehat{Q}_{S_1F_1}(\widetilde{Q}_{F_1F_1})^{-1}\widetilde{Q}_{F_1S_1}$ corresponds to transition intensities given by the rates entering a state in $\I_{F_1}$ from $\I_{S_1}$ times the exit probabilities to $\I_{S_1}$. As $\widetilde{Q}_{F_1S_1},\widetilde{Q}_{F_1F_1}$ have the transition intensities from the reactions $\mF$ and $\widehat{Q}_{S_1F_1}$ from $\mR\setminus\mF$, this corresponds exactly to the transition rates of the defined reduced RN of $\S$ \ref{subs_red_stoch}. Hence Theorem \ref{main_thm_stand} follows.

\section{Proof of Theorem \ref{main_thm_stand} for non-homogeneous SRNs}\label{proof_nonhom}

The proof follows the same steps as the proof of Theorem \ref{main_thm_stand}, hence it is enough to check that Assumptions \ref{ass2} and \ref{ass4} together with $\mF$ being proper are  sufficient for the assumptions made in \cite{yin2012continuous} to hold.
Recall that since the $\lam_{y \to y\p}(t,\cdot)$ satisfy Assumption \ref{ass1} for each $t\in [0,T]$, the state space decomposition of the ordinary Markov process for an RN and $Q_t$ for $t$ fix agree. Hence, by Assumption \ref{ass2} and Lemma \ref{transient_+},
any $x\in E$ with one molecule of a non-interacting species is transient in $\widehat{Q}_t$. Furthermore by Assumption \ref{ass4}, $\widetilde{Q}_t,\widehat{Q}_t$ are once continuous differentiable with Lipschitz derivative (\cite[Assumption A4.4(respectively,  A4.5)]{yin2012continuous}). With these observations we can conclude.

 \bibliographystyle{plain} 
    
 \bibliography{references}

\end{document}